\documentclass[a4paper]{article}

\usepackage{
amsmath,
amsthm,
amscd,
amssymb,
}
\usepackage{comment}
\usepackage{tikz}
\usetikzlibrary{positioning,arrows,calc}
\usepackage{xspace}

\usepackage{graphicx}

\usepackage{url}

\theoremstyle{plain}
\newtheorem{lemma}{Lemma}
\newtheorem{definition}{Definition}
\newtheorem{corollary}{Corollary}
\newtheorem{proposition}{Proposition}
\newtheorem{theorem}{Theorem}
\newtheorem{conjecture}{Conjecture}
\newtheorem{question}{Question}
\newtheorem{remark}{Remark}
\newtheorem{problem}{Problem}

\newtheoremstyle{derp}
{3pt}
{3pt}
{}
{}
{\upshape}
{:}
{.5em}
{}
\theoremstyle{derp}
\newtheorem{example}{Example}

\newcommand{\Z}{\mathbb{Z}}

\newcommand{\N}{\mathbb{N}}
\newcommand{\M}{\mathbb{M}}
\newcommand{\B}{\mathcal{L}}

\newcommand{\INF}{{}^\infty}

\newcommand{\QW}{\mathrm{Q}}

\newcommand{\CB}{\mathrm{CB}}

\newcommand{\OC}[1]{\overline{\mathcal{O}(#1)}}

\newcommand\xqed[1]{%
  \leavevmode\unskip\penalty9999 \hbox{}\nobreak\hfill
  \quad\hbox{#1}}
\newcommand\qee{\xqed{$\triangle$}}

\title{Decidability and Universality of Quasiminimal Subshifts}

\author{
Ville Salo \\
vosalo@utu.fi
}

\begin{document}
\maketitle

\begin{abstract}
We introduce the quasiminimal subshifts, subshifts having only finitely many subsystems. With $\N$-actions, their theory essentially reduces to the theory of minimal systems, but with $\Z$-actions, the class is much larger. We show many examples of such subshifts, and in particular construct a universal system with only a single proper subsystem, refuting a conjecture of [Delvenne, K\r{u}rka, Blondel, '05].
\end{abstract}


\section{Introduction}
\label{sec:Intro}

One of the most studied classes of subshifts\footnote{Here, subshifts are sets $X \subset S^M$ which are topologically closed, and also closed under the shift action of $M$ in the sense $X \cdot M \subset X$, where $M$ is a monoid and $S$ a finite set. In this paper, we consider the cases $M \in \{\N, \Z\}$.} in the literature is that of minimal subshifts. These are precisely the nonempty subshifts containing no proper nonempty (sub-)subshifts. Some reasons that these subshifts are of great interest are that every subshift contains a minimal subshift, and many natural examples of subshifts, such as those generated by primitive substitutions and those generated by Toeplitz sequences, are minimal. We direct the reader to \cite{Au88,Ku03} for a discussion of such systems. More generally, dynamical systems on compact metric spaces always contain minimal subsystems.

In \cite{DeKuBl06}, zero-dimensional dynamical systems with an $\N$-action and an effective presentation, called \emph{symbolic systems}, are studied from the point of view of computational universality. These systems generalize one-sided subshifts whose language is recursive, and also contain many other symbolic systems such as Turing machines, counter machines and tag systems. Given any finite clopen partition of a space and a Muller language of infinite words (one of the notions of regularity for infinite words) with partition elements as letters, the \emph{model-checking problem} is defined as the problem of checking whether one of the sequences in the language corresponds to a sequence of observations along an orbit. The model-checking problem for regular languages is the question of whether a finite sequence of observations in the language can be made in the system. A system is called \emph{decidable} if the model-checking problem for Muller languages is decidable, and \emph{universal} if its model-checking problem for regular languages is $\Sigma^0_1$-complete (that is, the halting problem many-one reduces to it). In particular, a universal system cannot be decidable, but a gap is left between the two definitions, so that both decidability and undecidability results are maximally strong.

In the case of subshifts, the model-checking problem for regular languages amounts to asking whether the intersection of the language of the subshift with a given regular language is empty, and we use this as the definition, omitting the details of computable presentations.

One of the main results of \cite{DeKuBl06} is that if a minimal system is computable\footnote{In their paper, computable means recursive, or having $\Delta^0_1 = \Sigma^0_1 \cap \Pi^0_1$ language, but see Theorem~\ref{thm:MinimalRecursive}. }, then it is decidable in the above sense. This shows in particular that many non-trivial things can be computed about minimal subshifts, as long as the forbidden patterns of the subshift can be enumerated. For example, given\footnote{It is not explicitly stated in \cite{DeKuBl06} that the algorithm is uniform in the description of the subshift, but this is clear from their proof.} a Turing machine enumerating the forbidden patterns of a (nonempty) minimal subshift $X$ over the alphabet $\{0, 1, 2\}$, an algorithm can check whether there exists a subword $w$ of a point of $X$ where the number of $1$s differs from the number of $0$s by more than 7 while simultaneously the number of $2$s in $w$ is not divisible by 5.

It is not particularly hard to show that if the forbidden words of $X$ can be enumerated and $X$ is minimal, then also the words that do occur in $X$ can be enumerated (see Theorem~\ref{thm:MinimalRecursive}), so that if such $w$ exists, a rather simple algorithm can find it. The algorithm of \cite{DeKuBl06} must be even smarter: if no $w$ with this list of properties exists, then after enumerating some finite number of forbidden patterns, the algorithm will state this fact.

The algorithm of \cite{DeKuBl06} applies more generally to systems where every proper subsystem has nonempty interior, and to systems whose limit set is a finite union of minimal systems.\footnote{See Section~\ref{sec:NActions} for an application of this result.} 
It is interesting to ask where the precise border of decidability lies, by extending the family of subshifts further, and to this end the authors also make the following conjecture.

\begin{conjecture}[\cite{DeKuBl06}]
\label{con:IntroManySubsystems}
A universal symbolic system has infinitely many minimal subsystems.
\end{conjecture}

Conversely, this conjecture would imply that a system with finitely many minimal subsystems has to be non-universal (even if not necessarily decidable). Note that this conjecture talks about infinitely many \emph{minimal} subsystems, but allows us to have any number of subsystems in general. It turns out that this is not a very strong condition, and as such, the conjecture is false.

\begin{proposition}
\label{prop:IntroOneMinimal}
There exists a recursive universal subshift with finitely many minimal subsystems.
\end{proposition}

The proof of this is given in Section~\ref{sec:Universal}, Proposition~\ref{prop:OneMinimal}. Our example is a $\Z$-subshift, and we interpret subsystems in the sense of being closed under the $\Z$-action, while subsystems in the sense of \cite{DeKuBl06} need to be closed under the induced $\N$-action only. It is easy to see that this does not change the number of minimal subsystems, although it may change the number of subsystems in general (see  Proposition~\ref{prop:Correspondence} for details), so that our example also provides an $\N$-system with the desired properties. Alternatively, one can directly modify our example to be one-sided.

The subshift is very simple, and we offer multiple variations of it. At its simplest, the number of minimal subsystems is one, and this subsystem is just a single point.\footnote{The one-point subshift is the simplest minimal subshift, but any minimal subshift can be used here, though naturally at the expense of countability.} The subshift is countable, and in fact contained in a countable sofic shift. Alternatively, the subshift can be contained in a countable SFT, although with slightly more minimal subsystems. The Cantor-Bendixson rank of the enveloping countable sofic shift is $4$. CB-rank $4$, for a countable sofic shift, means that each point contains at most $3$ disturbances to periodicity. 
In addition to the subshift being very simple, also the universality is very strong: not only are intersections with regular languages $\Sigma^0_1$-complete, but even the \emph{undirected halting problem}, the question of whether two given words $u$ and $v$ occur in the same point, is $\Sigma^0_1$-complete.

While $X$ has only finitely many minimal subsystems, it has infinitely many subsystems altogether. In fact, the set of subshifts of $X$ has the cardinality of the continuum (which is the maximal possible). In the abstract and introduction of \cite{DeKuBl06}, the authors state the conjecture in a weaker form, asking if a universal system should have infinitely many subsystems, without any mention of minimality. While this was presumably just meant as a shortened form of the statement of Conjecture~\ref{con:IntroManySubsystems}, it is a natural next question whether at least this is true.

\begin{question}
\label{q:IntroWeaker}
Must a universal symbolic system have infinitely many subsystems?
\end{question}

We name this class for easier reference.

\begin{definition}
A subshift is \emph{quasiminimal} if it has finitely many proper subshifts.
\end{definition}

The question for us is then whether a recursive quasiminimal subshift can be universal. One might guess that if $X$ has only finitely many subshifts, then it must be a quite simple extension\footnote{Here, we use the term extension in the sense of containment (monomorphisms), and not in the sense of factoring (epimorphisms). This usage is somewhat nonstandard in the theory of dynamical systems, but it is fitting for quasiminimal systems, since they are inductively built, in finitely many steps, from smaller quasiminimal systems by adding new points.} of the one or more minimal subshifts it necessarily contains, or perhaps even essentially just a union. It turns out that, unlike in Proposition~\ref{prop:IntroOneMinimal}, whether this is true depends on whether the acting monoid is $\N$ or $\Z$, that is, on whether our subshifts are one- or two-sided.\footnote{Strictly speaking, we could also consider $S^\Z$ with an $\N$-action (obtaining a rather unnatural definition of a subshift), but the convention is that $S^M$ uses the natural shift action of $M$.}

In the case of $\N$-actions, we show that the extensions of minimal subshifts to quasiminimal ones are rather trivial, and the answer to Question~\ref{q:IntroWeaker} is ``yes''.

\begin{theorem}
\label{thm:NDecidable}
A recursive quasiminimal $\N$-subshift is decidable.
\end{theorem}

This is shown in Corollary~\ref{cor:NDecidable}. The proof is quite short: we characterize this class, and show that it fits one of the decidability result proved in \cite{DeKuBl06}. Our main interest is in the case of $\Z$-actions, where the extensions can be quite complicated. For example, while primitive substitutions give rise to minimal subshifts, all aperiodic substitutions satisfying a technical property give rise to quasiminimal ones, see Proposition~\ref{prop:SubstitutionsQuasiminimal}. These examples already show that quasiminimality is indeed quite different from minimality. In fact, in this setting, the answer to Question~\ref{q:IntroWeaker} is ``no''.

\begin{theorem}
\label{thm:IntroQuasiminimal}
There exists a recursive universal quasiminimal $\Z$-subshift.
\end{theorem}

While this answers Question~\ref{q:IntroWeaker} completely in the case of $\Z$-actions, it turns out to be a bit more intricate than Conjecture~\ref{con:IntroManySubsystems} -- for example, we will see that the halting problem (whether a clopen set is reachable from another one) is decidable for a recursive quasiminimal subshift. In fact, given any finite tuple of words, an algorithm can check whether there exists a point in the subshift containing those words, in order (see Theorem~\ref{thm:TupleTheorem}). In particular (though not quite equivalently) we can solve the model-checking problem for regular languages of the form $w_1 S^* w_2 S^* \cdots S^* w_k$ where $S$ is the alphabet of the subshift and the $w_i$ are arbitrary words. This class of regular languages is related to the piecewise testable regular languages.

On the other hand, in Theorem~\ref{thm:TransitiveLocallyTestableLanguages} we show that the halting problem along a clopen set is $\Sigma^0_1$-complete, that is, the halting problem where we restrict the path between the given clopen sets to stay in a third clopen set. This corresponds to the fact that the model-checking problem is undecidable for locally testable languages, a well-known subclass of regular languages.

We also study countable quasiminimal subshifts. For these, we believe the model-checking problem for locally testable languages is in fact decidable (Conjecture~\ref{con:LocallyTestableDecidable}), which would show a great difference between the countable and uncountable cases. We prove the undecidability of the model-checking problem for piecewise testable languages and what we call renewal languages (languages obtained as concatenations of a set of words, starting and ending in a given word), in Theorem~\ref{thm:CountablePiecewiseTestableLanguages} and Theorem~\ref{thm:CountableCodedLanguages}. The model-checking problem for piecewise testable languages is dealt with by showing that the ``halting problem along a clopen set except on exactly $k$ steps'' (which we call the counting problem) is $\Sigma^0_1$-complete, and the model-checking problem for renewal languages by showing that the ``halting problem in time $m \bmod k$'' is $\Sigma^0_1$-complete.

An obvious question is how many subshifts exactly do we need to have to achieve universality, since the result of \cite{DeKuBl06} shows we need at least one. In the uncountable case, our example has only one minimal subshift (which can be chosen to be any infinite minimal subshift). This is clearly optimal. For countable quasiminimal subshifts, the number of nontrivial (nonempty and non-full) subshifts is two in our example for renewal systems, and we show this to be optimal in Proposition~\ref{prop:3IsDecidable}, since countable systems with only one subsystem are in fact decidable in general. For the particular case of piecewise testable languages, the optimal number of subshifts is shown to be~$4$.

In addition to proving a small gap between decidability and undecidability, these constructions are interesting as first explicit examples of quasiminimal $\Z$-subshifts. In particular, they show that a quasiminimal $\Z$-subshift can be quite far from a union of minimal $\Z$-subshifts, and even countable quasiminimal $\Z$-subshifts can be rather complicated objects. In fact, these constructions are in some sense representative of the general case, in that a full characterization of quasiminimal subshifts can be obtained by generalizing and iterating the constructions of this article. We defer the full characterization to a later work \cite{Quasiminimals}, and prove here only enough structural results to solve our decidability problems of interest in Section~\ref{sec:Decidability}.

We also show in Theorem~\ref{thm:ContextFree} that the model-checking problem of context-free languages can be undecidable even for minimal subshifts.

\section{Definitions and basic observations}


\subsection{More or less standard definitions}

\begin{remark}
In this article, $\N = \{0,1,2,3, \ldots\}$, and words are $0$-indexed. We give the definitions of subshifts and related concepts for $\Z$-actions in this section. It is easy to modify them in the case of $\N$-actions for the few places where we need them, the crucial difference being that we do not require surjectivity for $\N$-actions. A subset of $S^\Z$, when considered a dynamical system, will always be considered a system with a $\Z$-action given by the shift, and a subset of $S^\N$ is an $\N$-system.
\end{remark}

\begin{definition}
A \emph{subshift} is a compact subset $X$ of $\N^\Z$, where $\N$ has the discrete topology and $\N^\Z$ the product topology, and which is shift-invariant, in the sense that $\sigma(X) = X$, where $\sigma : \N^\Z \to \N^\Z$ is the shift map $\sigma(x)_i = x_{i+1}$.
\end{definition}

For standard references see \cite{LiMa95,Ku03,Ki87}.

Note that this indeed corresponds (up to symbol renaming) to the usual definition where a finite set is used in place of $\N$, and closed shift-invariant subsets are called subshifts. Namely, for any finite $S \subset \N$, $X \subset S^\Z$ is a subshift if and only if $X$ is closed and shift-invariant, since $S^\Z$ is compact and Hausdorff, so that its closed subsets are precisely its compact subsets. On the other hand, no shift-invariant set $X \subset \N^\Z$ containing infinitely many letters is compact, since the open cover $\{[a]_0 \;|\; a \in \N\}$ has no finite subcover. Accordingly, when we write that $X \subset S^\Z$ is a subshift, we imply that $S$ is chosen to be some finite alphabet. The induced topology of $S^\Z$ for finite $S$ is generated by the \emph{cylinders} $[w]_i = \{x \in S^\Z \;|\; x_{[i,i+|w|-1]} = w\}$ for $w \in S^*$ and $i \in \Z$, which are clopen in $S^\Z$. The clopen sets are precisely the finite unions of cylinders, and can all be represented as $[C]_i = [w_1]_i \cup [w_2]_i \cdots \cup [w_k]_i$ for some $i \in \Z$, $n \in \N$ and $C = \{w_1, \ldots, w_k\} \subset S^n$. We write $[C] = [C]_0$.

For words $u, v \in \N^*$ (that is, two finite words over $\N$), we write $uv$ for their concatenation. For words $u, v, v', w \in \N^*$, we write $\INF u v. v' w\INF$ for the point $x \in \N^\Z$ with $x_{[0,|v'|-1]} = v'$, $x_{[-|v|,-1]} = v$, and for all $i \in \N$,
\[ x_{[-|v|-(i+1)|u|,-|v|-1-i|u|]} = u \]
and
\[ x_{[|v|+i|w|, |v|+(i+1)|w|-1]} = w. \]
Note that coordinate $0$ is to the right of the decimal point. Often, however, this position is not relevant, and the decimal point is omitted. We sometimes write $\cdots u . v \cdots$ when  the continuations from $u$ to the left and from $v$ to the right are easy to guess. For example, $x = \cdots 0000.0123 \cdots$ is the point $x \in \N^\Z$ where for all $i \in \N$ we have $x_{-i} = 0$ and $x_i = i$. Similar notations are used for points $x \in \N^\N$. 

A \emph{central pattern} of $x \in \N^\Z$ is a word $x_{[-n, n]} \in \N^*$ for some $n \in \N$. For $x \in \N^\Z$, we write $u \sqsubset x$ if $x_{[i,i+|u|-1]} = u$ for some $i \in \Z$. We also write $u \sqsubset v$ for two words $u, v \in \N^*$ if $u$ is a subword of $v$. This extends to sets $A \subset \N^\Z$ by $u \sqsubset A$ if $u \sqsubset x$ for some $x \in A$.

A subshift can also be defined by a set $F \subset S^*$ of \emph{forbidden words}, in the sense that if $X \subset S^\Z$ is a subshift, then there exists $F \subset S^*$ such that
\[ X = \{ x \in S^\Z \;|\; \forall u \in F: u \not\sqsubset x \}. \]
If $F$ can be taken to be finite, then $X$ is called an \emph{SFT}, and if it can be taken to be a regular language (a language accepted by a finite-state automaton \cite{HoMoUl06}), then $X$ is called a \emph{sofic shift}. Every SFT is sofic, but the converse is not true. The \emph{SFT approximation of order $k$} of a subshift $X \subset S^\Z$ is the SFT $Y_k$ whose forbidden patterns are precisely the words $S^k \setminus \B_k(X)$. Thus, for all $k$ and $m \geq k$, we have $\B_k(Y_m) = \B_k(X)$.

If $X$ and $Y$ are subshifts, a continuous function $f : X \to Y$ between them is called a \emph{morphism} if $f \circ \sigma = \sigma \circ f$. A surjective morphism is called \emph{factor map}, and a bijective one is called a \emph{conjucagy}, or a \emph{recoding}.

For a point $x \in \N^\Z$, we write $\B_n(x)$ for length-$n$ subwords of $x$, that is, $\B_n(x) = \{u \in \N^n \;|\; u \sqsubset x \}$, and $\B(x) = \bigcup_n \B_n(x)$. This is called the \emph{language} of $x$. For any set $A \subset \N^\Z$, we write $\B_n(A) = \bigcup_{x \in A} \B(x)$, and define $\B(A)$ as before. A subshift $X \subset S^\Z$ is uniquely determined by its language $\B(x)$, which is always \emph{factor-closed} ($vwv' \in \B(X) \implies w \in \B(X)$) and \emph{extendable} ($u \in \B(X) \implies \exists a, b \in S: aub \in \B(X)$).

Accordingly, for a factor-closed and extendable language $L$, we write $\B^{-1}(L)$ for the unique subshift $X$ with $\B(X) = L$. We extend this notation to all extendable languages $L$ by $\B^{-1}(L) = \B^{-1}(\mathrm{Fact}(L))$, where $\mathrm{Fact}(L) = \{w \;|\; \exists u \in L : w \sqsubset u\}$ is the \emph{factor closure} of $L$. For example, we can write the (sofic) \emph{even shift} \cite{LiMa95} with forbidden patterns $1(00)^*01$ as $\B^{-1}((1(00)^*)^*)$, and the (SFT) \emph{golden mean shift} with the single forbidden pattern $11$ as $\B^{-1}((100^*)^*)$. The \emph{orbit} of a point $x \in \N^\Z$ is $\mathcal{O}(x) = \{\sigma^i(x) \;|\; i \in \Z\}$. The \emph{orbit closure} $\OC{x}$ of $x \in S^\Z$ is the (topological) closure of the orbit of $x$ in $S^\Z$.\footnote{See Section~\ref{sec:Ruler} for the case where $x$ has infinite language.} It is easy to see that $\OC{x}$ is the smallest subshift containing $x \in S^\Z$.

A \emph{substitution} is a function $\tau : A \to B^*$, where $A, B \subset \N$. In our applications, usually $A = \N$, and $B = S$ where $S$ is the alphabet of a subshift being discussed. For $\phi \in \N^\Z$ and $\tau : \N \to S^*$, we write
\[ \tau(\phi) = \cdots \tau(\phi_{-2}) \tau(\phi_{-1}) . \tau(\phi_0) \tau(\phi_1) \tau(\phi_2) \cdots. \]



A \emph{weakly transitive point (for $X$)} is a point $x \in X$ such that $w \sqsubset x$ holds for all $w \sqsubset X$. If $X$ contains a weakly transitive point, we say $X$ is \emph{weakly transitive}. A \emph{doubly transitive point (for $X$)} is a point $x \in X$ such that every word $w \sqsubset X$ occurs infinitely many times in both tails of $x$. If $X$ contains a doubly transitive point, then we say $X$ is \emph{transitive}. A weakly transitive subshift need not be transitive. For example, a transitive subshift is either finite or uncountable (since a countable subshift contains an \emph{isolated point}, that is, a singleton open set \cite{BaDuJe08}), but the infinite countable subshift $\B^{-1}(0^*10^*) = \overline{\mathcal{O}(\INF 0 1 0 \INF)}$, called the \emph{sunny-side-up subshift} is weakly transitive since it is the orbit closure of the point $\INF 0 1 0 \INF$.

A subshift $X \subset S^\Z$ is \emph{minimal} if we have $Y \in \{\emptyset, X\}$ for all subshifts $Y \subset X$. This is equivalent to the fact that $X$ is \emph{uniformly recurrent}, that is,
\[ w \sqsubset X \implies \exists n: \forall u \in \B_n(X): w \sqsubset u. \]
We write $a^\Z$ for the point $x$ with $\forall i \in \Z: x_i = a$. We have $\OC{a^\Z} = \{a^\Z\} = \B^{-1}(a^*)$. This is an example of a (finite) minimal subshift. Typically, minimal subshifts are uncountable. Such examples can be found in \cite{Ku03}.


There is a canonical way to remove the isolated points of a space (in our case, a subshift): The \emph{Cantor-Bendixson derivative} of a subshift $X$ is the subshift $X^{(1)}$ obtained by forbidding all words $w$ from $X$ such that $[w]_0$ is a singleton. That is, $X' = X^{(1)}$ is $X$ without its isolated points. This process can be repeated up to any ordinal in an obvious way by transfinite induction, and the \emph{Cantor-Bendixson rank} of a subshift $X$ is the least ordinal $\lambda$ such that $X^{(\lambda)} = X^{(\lambda + 1)}$, where $X^{(\beta)}$ denotes the $\beta$th Cantor-Bendixson derivative of $X$. The Cantor-Bendixson rank of $X$ is denoted by $\CB(X)$, and we call the set $X^{(\CB(X))}$ the \emph{Cantor-Bendixson center} of $X$. The Cantor-Bendixson center is the unique maximal perfect subspace of $X$, which is automatically a subshift. A subshift $X$ is countable if and only if its Cantor-Bendixson center is empty. We refer to \cite{BaDuJe08} for details.

We need the first few levels of the \emph{arithmetical hierarchy} \cite{Sa90}. Namely, $\Sigma^0_1$ is the set of recursively enumerable languages, that is, languages such that there is an algorithm that halts on words in the language, but not on the ones outside it. The set of complements of $\Sigma^0_1$ languages is $\Pi^0_1$. A language in $\Delta^0_1 = \Sigma^0_1 \cap \Pi^0_1$ is called \emph{recursive}. Properties that are recursive when encoded into languages (in some natural way, usually safe to leave implicit) are called \emph{decidable}.

We say that a subshift is $\Pi^0_1$ or $\Sigma^0_1$ if its language is. There are other possible definitions for $\Pi^0_1$ subshifts, and we give a few below.

\begin{lemma}
For a subshift $X \subset S^\Z$, the following are equivalent:
\begin{itemize}
\item $X$ is $\Pi^0_1$.
\item The set $S^* \setminus \B(X)$ is $\Sigma^0_1$.
\item There exists a Turing machine $M$ which enumerates an infinite list of words $w_1, w_2, \ldots$ such that if $x \in S^\Z$, then $x \in X$ if and only if $w_i \not\sqsubset x$ for any $i \in \N$.
\item There exists a Turing machine $M$ such that the machine $M^x$ ($M$ with oracle $x \in S^\Z$) halts if and only if $x \notin X$.
\end{itemize}
\end{lemma}

We call a point $x \in S^\Z$ \emph{computable} if there is an algorithm that, given $i \in \Z$, computes $x_i$. It is well-known that not all $\Pi^0_1$ subshifts contain computable points, see for example \cite{CeDaToWy12}.


\subsection{The ruler sequence}
\label{sec:Ruler}

The set $\M = \N \cup \{\infty\}$ is the \emph{one-point compactification} of $\N$, where $U \subset \M$ is open if either $U \subset \N$, or $\M \setminus U$ is finite.

We give the space $\M^\Z$ the product topology, so that it becomes compact as well. A sequence $\phi \in \N^\Z$ is called \emph{recurrent} if every word that appears in $\phi$ appears infinitely many times in both directions. As usual, if every word appears with bounded gaps, we say $\phi$ is \emph{uniformly recurrent}. We say $\phi$ is \emph{Toeplitz} if for all $i \in \Z$ there exists a \emph{period} $p$ such that $\phi_i = \phi_{i+kp}$ for all $k$. It is easy to see that a Toeplitz sequence is uniformly recurrent. We say $\phi$ has \emph{unique singularities} if for each $m \in \N$ there exists $k \in \N$ such that
\[ \phi_i \geq k \implies \forall j \in [i-m,i+m] \setminus \{i\}: \phi_j \leq k. \]
We give $\M^\Z$ the shift action, and for each $\phi \in \M^\Z$ obtain a compact space $\overline{\mathcal{O}(\phi)}$ analogously to the case of finite alphabets.\footnote{We note that this is not the usual way to compactify subshifts with infinite alphabets.}

It is easy to verify that $\phi$ is recurrent if and only if $\phi$ is doubly transitive in $\overline{\mathcal{O}(\phi)}$, and, more importantly, by Theorem~7 in \cite{Au88} it is uniformly recurrent if and only if $\overline{\mathcal{O}(\phi)}$ is minimal. The unique singularities property for $\phi$ means precisely that each point of $\overline{\mathcal{O}(\phi)}$ contains at most one occurrence of $\infty$, as one can verify by a compactness argument. Of course, $\phi$ contains only finitely many distinct symbols if and only if $\infty \not\sqsubset \overline{\mathcal{O}(\phi)}$ if and only if $\overline{\mathcal{O}(\phi)}$ is a subshift. In particular, if there are finitely many symbols, having unique singularities is trivial, and recurrence and uniform recurrence correspond to the usual notions. The same observations hold for the space $\M^\N$.

We note that neither of uniform recurrence and unique singularities implies the other, as one can easily verify by examples.

We will essentially just need one recurrent sequence with unique singularities in our constructions. The proofs are particularly easy for the \emph{ruler sequence} $\phi \in \N^\N$,
\[ 0\;1\;0\;2\;0\;1\;0\;3\;0\;1\;0\;2\;0\;1\;0\;4\;0\;1\;0\;2\;0\;1\;0\ldots,\]
that is, $\phi = n_0 n_1 n_2 \ldots$ where $n_i$ is the largest integer $k$ with $2^k | i+1$. This is sequence A007814 in the OEIS database \cite{OEISruler}.

Note that for all $j$ such that $j \sqsubset \phi$ we have that $\{i \in \N \;|\; \phi_i = j\}$ is a semilinear set, that is, a finite union of translates of submonoids of $\N$. More precisely, for the ruler sequence $\phi$, we have
\[ \{i \in \N \;|\; \phi_i = j\} = \{n2^{j+1} + 2^j - 1 \;|\; n \in \N\}. \]
It follows that the ruler sequence is Toeplitz (thus uniformly recurrent) and has unique singularities.

We note some basic combinatorial properties of the subwords of the ruler sequence $\phi$, whose proofs we omit. If $\phi_i = \ell = \phi_j$ where $i < j$, then there exists $k \in [i, j]$ such that $\phi_k = \ell+1$. It follows that if $w \sqsubset X = \overline{\mathcal{O}(\phi)}$, then there exists a unique maximal symbol $w_i = k$ in $w$. For each $k \in \N$, there is a unique maximally long word $u \in \B(X)$ with $u_{2^k-1} = k$ and $\forall i: u_i \leq k$, and this $u$ satisfies $|u| = 2^{k+1}-1$.

Given a word $w \in \B(X)$ with $\infty \not\sqsubset w$, all possible extensions to longer words of $X$ are obtained by the following process: First, take the maximally long word $u \in \B(X)$ with $\max_i u_i = k = \max_i w_i$. We have $u = vwv'$ for a unique choice of $v, v'$, and $v$ ($v'$, resp.) is the unique extension of $w$ in $\B(X)$ of length $|v|$ to the left (of length $|v'|$ to the right, resp.). If $w = u$, then the extensions of $w$ to words in $\B(X)$ of the form $awb$ with $a, b \in \N$ are precisely those of the forms $(k+1) w \ell$ and $\ell w (k+1)$, where $\ell > k+1$.

For example, consider the word $w = 3$. To obtain its extensions, we first deterministically extend it to $u = 010201030102010$. Here, we have $v = v' = 0102010$. We now choose a side on which to write $4$. Choosing the left side, we obtain $4010201030102010$. On the other side, we can choose any number greater than or equal to $5$. Choosing $5$ and performing the deterministic extensions, we obtain the word
\[ 010201030102010\;4\;0102010\; 3\;0102010\;5\;0102010301020104010201030102010 \]
(where some spaces have been added for clarity). This is the unique maximally long word with maximal symbol 5.

Although $\phi \in \M^\N$, we can take its limit points also in $\M^\Z$ in the obvious way, and in the orbit closure of $\phi$ in $\M^\Z$, there is a computable point $\psi$ where $\infty$ does not occur. That is, $\psi \in \N^\Z$. One such example is
\[ \psi = \ldots01020103010201040102.010301020105010201030102010\ldots, \]
which is constructed by starting with $0$ and repeating the following steps, keeping the original $0$ at coordinate $0$:
\begin{itemize}
\item Extend the word on the right by the smallest symbol not yet found in it.
\item Extend the word deterministically, as much as possible.
\item Extend the word on the left by the smallest symbol not yet found in it.
\item Extend the word deterministically, as much as possible.
\end{itemize}

\subsection{Model-checking and halting problems}

The following definitions are those from \cite{DeKuBl06}, adapted to the case of subshifts.\footnote{In \cite{DeKuBl06}, these definitions are given not just for subshifts, but for the more general class of "effective symbolic systems", which in the case of subshifts correspond to ones with a recursive language. Although we do not require recursivity a priori, our examples have recursive languages.}

\begin{definition}
Let $X \subset S^\Z$ be a subshift. The \emph{model-checking problem of $X \subset S^\Z$ for Muller automata} is to decide, given a language $Y \subset S^\Z$ of infinite words described by a Muller automaton, whether $X \cap Y \neq \emptyset$. The \emph{model-checking problem of $X$ for finite-state automata} is to decide, given a regular language $L \subset S^*$ described by a finite-state automaton, whether $\B(X) \cap L \neq \emptyset$.
\end{definition}

These are chosen as the natural problems for defining decidability and universality in \cite{DeKuBl06}, and we take the same view. 

\begin{definition}
A subshift $X \subset S^\Z$ is \emph{decidable} if its model-checking problem for Muller automata is decidable. It is \emph{universal} if its model-checking problem for finite-state automata is $\Sigma^0_1$-complete.
\end{definition}

It is worth noting that the complexity of a subshift is often defined to be the complexity of its language, or coincides with it. For example, a \emph{$\Pi^0_1$ subshift} is a subshift whose language is $\Pi^0_1$. For decision problems, ``decidable'' is a synonym for ``recursive'', but ``decidable subshifts'' in the sense of the previous definition form a much smaller class than ``recursive subshifts'' (decidability implies that the language is recursive, but not vice versa). We will refer to subshifts whose language is recursive exclusively as \emph{recursive subshifts}.

Another difference in the definition of \cite{DeKuBl06} is that the regular languages are over clopen sets instead of letters of the alphabet. One way to formulate this is that, instead of asking whether $\B(X) \cap L \neq \emptyset$ for a given regular language $L$, we ask whether $\B(\phi(X)) \cap L \neq \emptyset$ for a factor map $\phi : X \to Y$ where $Y \subset T^\Z$ is a subshift and $L \subset T^*$ is a regular language. If $X$ is recursive, it is easy to see that this does not change the definitions of universality and decidability, since a finite-state machine is smart enough to look at the subshift through a factor map (we skip the easy proof).

In practise, to show undecidability, one (many-one) reduces the halting problem of Turing machines to the model-checking problem for regular languages, and usually the reduction will be far from surjective, in the sense that the regular languages produced by the reduction are of a specific form. By considering these different forms, one can get more fine-grained information on the degree to which the subshift is decidable. 
For this general definition, it is useful to include clopen sets, since, when restricting to small classes of regular languages, we may not be able to simulate factor maps (or even conjugacies), leading to dynamically unnatural decision problems.

\begin{definition}
Let $\mathcal{C} = \bigcup_i \mathcal{C}_i$ be a class of regular languages, where each $\mathcal{C}_m$ is over the alphabet $[1, m] \subset \N$. The \emph{model-checking problem for $\mathcal{C}$ and a subshift $X$} is, given a language $L$ in $\mathcal{C}_m$ for some $m$ and a clopen partition $X = [C_1] \dot\cup \cdots \dot\cup [C_m]$, to decide whether there exists a word $w \in L$ and point $x \in X$ such that $\sigma^i(x) \in [C_{w_i}]$ for all $i \in [0, |w|-1]$.
\end{definition}

As noted above, this definition is equivalent to checking $\B(\phi(X)) \cap L \neq \emptyset$ for factor maps $\phi$, and universality of a subshift means that its model-checking problem for the whole class of regular languages is $\Sigma^0_1$-complete.

\begin{definition}
The \emph{starfree languages over $[1, m]$} are the closure of finite languages over $[1, m]$ under concatenation, complementation (with respect to $[1, m]^*$) and union.\footnote{In other words, they are defined by regular expressions without Kleene star.} By an \emph{elementary piecewise testable language over $[1, m]$}, we mean a language of the form $a_1 S^* a_2 S^* a_3 S^* \cdots S^* a_k$ where $S = [1, m]$, $a_i \in S$ for all $i \in [1, k]$, and a \emph{piecewise testable language over $[1, m]$} is any Boolean combination of elementary piecewise testable languages. A \emph{local language over $[1, m]$} is a language of the form $(AS^* \cap S^*B) \setminus S^*FS^*$, where $A, B \subset S$ and $F \subset S^2$. By a \emph{renewal language} we mean one of the form $u(w_1 + ... + w_k)^*v$ where $u, v, w_1, \ldots, w_k \in S^*$.
\end{definition}

In the case of starfree languages and piecewise testable languages, it is important to parametrize the class of languages with the alphabet $[1, m]$, and consider languages over $[1, m]$ when the space is partitioned into $m$ clopen sets, because the alphabet $S$ plays a special role in the definitions -- for local languages and renewal languages, this is not necessary. Because $S^* = \emptyset^C$, piecewise testable languages and locally testable languages are both starfree.

The starfree languages are a well-studied class of regular languages, and were also essentially studied in the context of dynamical systems in \cite{DeKuBl05} (an earlier version of \cite{DeKuBl06}), as definability in temporal logic corresponds to starfreeness \cite{Th90} (although this connection was not made explicit). ``Elementary piecewise testable'' is not a standard term, but piecewise testable languages are again well-studied (although there are many variations of the basic idea in the literature). Local languages are important in the theory of regular languages, as they are to regular languages what vertex shifts are to sofic shifts. We could have equally well used \emph{locally testable languages}, which correspond more directly to SFTs, but this is mostly irrelevant for dynamical considerations. We do not know of an occurrence of renewal languages in the literature, but they are an obvious ``language version'' of the well-studied class of sofic shifts called \emph{renewal systems}.\footnote{One can imagine many other such language versions as well, and we do not claim this to be the ``correct'' one -- for the purpose of this article, it is mainly a mnemonic.}

Another approach to decidability is to consider problems of reachability between clopen sets. We talk about the following problems:

\begin{definition}
In the following problems, defined for a subshift $X$, we are given clopen sets $[C], [D], [E], [F]$ and numbers $k, m$, or a subset of these inputs, and need to decide whether there exists a point $x \in X$, $j \in \Z$ and $A \subset [1, j-1]$ with $|A| = k$, or a subset of these, with particular properties.
\begin{itemize}
\item In the \emph{undirected halting problem}, the properties $x \in [C]$ and $\sigma^j(x) \in [D]$.
\item In the \emph{halting problem}, the properties $j \in \N$, $x \in [C]$ and $\sigma^j(x) \in [D]$.
\item In the \emph{halting problem along a clopen set}, the properties $j \in \N$, $x \in [C]$, $\sigma^j(x) \in [D]$ and $\forall i \in [1, j-1]: \sigma^i(x) \in E$.
\item In the \emph{counting problem}, the properties $j \in \N$, $x \in [C]$, $\sigma^j(x) \in [D]$, $\forall i \in [1, j-1] \setminus A: \sigma^i(x) \in [E]$ and $\forall i \in A: \sigma^i(x) \in [F]$.
\item In the \emph{modular halting problem}, the properties $j \in \N$, $j \equiv k \bmod m$ with $x \in [C]$ and $\sigma^j(x) \in [D]$.
\item In the \emph{modular halting problem along a clopen set}, the properties $j \in \N$, $j \equiv k \bmod m$ with $x \in [C]$, $\sigma^j(x) \in [D]$ and $\forall i \in [1, j-1]: \sigma^i(x) \in E$.
\end{itemize}
\end{definition}

The halting problem was also defined in \cite{DeKuBl06}, and for Turing machines, considered as a dynamical system (with moving tape or moving head, see \cite{Ku97}), it essentially corresponds to the usual halting problem. It is easy to see that the halting problem is a special case of the model-checking problem for elementary piecewise testable languages.

The halting problem along a clopen set is easily seen to be a special case of the model-checking problem for local languages (and thus starfree languages) and renewal languages. It is also a special case of that of piecewise testable languages: For any $a_1, a_2, a_3 \in S$, the language $L = \{a_1, a_2, a_3\}^*$ is piecewise testable, and the language
\[ L' = a_1 S^* \cap S^* a_2 \cap (S^+ \{a_1, a_2\} S^+)^C \]
of words beginning with $a_1$, ending in $a_2$ and not containing other occurrences of these symbols is piecewise testable, because it is the intersection of
\[ a_1 S^* \cap (S^+ a_1 S^*)^C = \bigcap_{a \neq a_1} (S^* a S^* a_1 S^*)^C \cap S^* a_1 S^* \]
and $S^* a_2 \cap (S^* a_2 S^+)^C$. The halting problem along a clopen set is, up to choosing a suitable clopen partition, the model-checking problem for the single piecewise testable language $L \cap L'$.

The counting problem is a special case of the model-checking problem for piecewise-testable languages as well: For $a_1, a_2, a_3, a_4 \in S$, the language $L \cap L' \cap L_k \cap L_{k+1}^{C}$ is piecewise-testable, where $L = \{a_1, a_2, a_3, a_4\}^*$, $L'$ is as above, and $L_i = (S^* a_4)^i S^*$ is the language of words containing at least $i$ occurrences of the symbol $a_4$, which is piecewise testable by form.

Finally, the modular halting problem (along a clopen set or not) is a special case of the model-checking problem for renewal languages: For $a_1, a_2, a_3 \in S$, set $u = a_1 a_3^k$, $v = a_2$ and $w_1 = a_3^m$. Then $uw_1^*v$ is a renewal language. We note that this is not a special case of the model-checking problem for starfree languages, and indeed should not be: a well-known characterization of starfreeness of a language is that it is recognized by a finite state machine that does no modular counting \cite{McPa71}.

\subsection{The generating order}

The following object is very useful in studying subshifts, and was introduced in \cite{BaDuJe08} in the context of countable multidimensional SFTs.

\begin{definition}
The \emph{pattern preorder} of a subshift $X \subset S^\Z$ is the order $x \leq y \iff \B(x) \subset \B(y)$.
\end{definition}

We will, instead, talk about the \emph{subpattern poset}, meaning the preordered set with elements $X$ and the pattern preorder as the preorder. If $x \leq y \leq x$, we write $x \sim y$.

We wish to extend this relation to finite words as well. An obvious way to do this would be to define $u \leq v \iff \B(u) \subset \B(v)$, so that $u \leq v \iff u \sqsubset v$. However, the following definition is more useful:

\begin{definition}
The \emph{generating (pre-)order} of a subshift $X$ has elements $\B(X)$ and preorder $u \leq_X v \iff (\forall x \in X: v \sqsubset x \implies u \sqsubset x)$.
\end{definition}

By compactness, $u \leq_X v$ means that there exists $k \in \N$ such that
\[ \forall x \in X: x_{[0, |v|-1]} = v \implies u \sqsubset x_{[-k,|v|+k-1]}. \]
We again talk about the generating poset to mean the preordered set $\B(X)$ with the generating preorder.

Abusing notation, for a word $u$ and a point $x$, we write $x \leq_X u$ if $v \leq_X u$ for all $v \sqsubset x$. When $X$ is clear from context, in particular when $u$ and $v$ are explicitly chosen from $X$, we write $u \leq v$ for $u \leq_X v$.

\begin{lemma}
\label{lem:LeqSemidecidable}
Let $X$ be a $\Pi^0_1$ subshift. Then, given $u, v$ it is semidecidable whether $u \leq_X v$. In fact, there is an algorithm that, given $u, v$ with $u \leq_X v$, computes $h_{u,v}$ such that
\[ |w| \geq |v|+2h_{u,v} \wedge w_{[h_{u,v}, h_{u,v}+|v|-1]} = v \implies u \sqsubset w. \]
\end{lemma}

\begin{proof}
If $u \leq v$, then all long enough legal patterns of $X$ containing $v$ in the center contain $u$ as well. By compactness, the same is true for some SFT approximation of $X$, which we eventually find by enumerating forbidden patterns of $X$. This yields $h_{u,v}$ as a side-product. 
\end{proof}

The poset induced by the $\leq_X$-order on words is not a conjugacy invariant. However, $\leq_X$ extends naturally to clopen sets by $[C] \leq_X [D] \iff x \in [D] \implies \mathcal{O}(x) \cap [C] \neq \emptyset$, and the poset of clopen sets is obviously conjugacy invariant. It is a direct corollary of the previous lemma that $[C] \leq_X [D]$ is semidecidable for given cylinders $[C]$ and $[D]$.

\section{Quasiminimality and undecidability}
\label{sec:Universal}

In this section, we present our constructions of universal subshifts mentioned in Section~\ref{sec:Intro} and prove their correctness. The decidability results complementing them are given in Section~\ref{sec:Decidability}. To get up to speed, we begin with some simpler examples. First, we give a few examples of quasiminimal but non-minimal subshifts which are in a sense already well-known: subshifts generated by letter-to-word substitutions. Then, we give a universal subshift with finitely many minimal subshifts, but infinitely many subshifts in total, already proving Proposition~\ref{prop:IntroOneMinimal} and refuting Conjecture~\ref{con:IntroManySubsystems}. Then, we present our main constructions: universal quasiminimal subshifts. We conclude this section with a minimal subshift whose model-checking problem for context-free languages is $\Sigma^0_1$-complete.

\subsection{Non-universal quasiminimal examples}
\label{sec:Substitutions}

In this section, we consider classical symbol-to-word substitutions on a finite alphabet $S$. It is well-known that a \emph{primitive} substitution, that is, a substitution $\tau : S \to S^*$ such that
\[ \exists n: \forall a, b \in S: b \sqsubset \tau^n(a), \]
generates a minimal subshift (in the sense defined below). A non-primitive substitution does not necessarily generate a minimal subshift, but it usually\footnote{Below, we give a proof under some assumptions on the substitution.} generates a quasiminimal subshift. We show some examples, and sketch the proofs of their quasiminimality. These examples are very similar to the universal examples we construct later.

\begin{definition}
Let $\tau : S \to S^*$ be a substitution, define a subshift $X_\tau \subset S^\Z$ that is \emph{generates} by
\[ x \in X_\tau \iff \forall j,k \in \N: \exists a \in S, \ell \in \N: x_{[j,k]} \sqsubset \tau^\ell(a). \]
\end{definition}

First, we give an example of a countable quasiminimal system which is not a union of minimal systems.

\begin{example}
\label{ex:CountableSubstitution1}
Let $\tau$ be the substitution $(0 \mapsto 0; 1 \mapsto 010)$. Then we have $\tau^n(0) = 0$ and $\tau^n(1) = 0^n 1 0^n$ for all $n \in \N$. Clearly, the substitution then generates the sunny side up subshift $\OC{^\infty 0 1 0 ^\infty}$. Since this subshift is the union of the orbits of $^\infty 0 1 0 ^\infty$ and $^\infty 0 ^\infty$, it is countable, and its only proper nontrivial subsystem is $\B^{-1}(0^*)$. \qee
\end{example}

The points need not be eventually periodic:

\begin{example}
\label{ex:CountableSubstitution2}
Let $\tau$ be the substitution $(0 \mapsto 00; 1 \mapsto 11; 2 \mapsto 20; 3 \mapsto 2301)$. Clearly, the only points where infinitely many central patterns are subwords of $\tau^n(0)$, $\tau^n(1)$ or $\tau^n(2)$ are $0^\Z$ and $1^\Z$. Thus, we only need to find out the limit points of $\tau^n(3)$. We have
\[ \tau^n(3) = 2 0^{n-1} 2 0^{n-2} 2 \cdots 2 0 2 \; 3 \; 0 1 00 11 0^{2^2} 1^{2^2} 0^{2^3} 1^{2^3} \cdots 0^{2^{n-1}} 1^{2^{n-1}}. \]
Since there is only one occurrence of the symbol $3$ in this word, there is a unique point in $X_\tau$ where $3$ occurs. In a proper subsystem, then, a finite subset of the symbols $\{0, 1, 2\}$ occurs. The proper subshifts of $X_\tau$ can be seen to be the countable sofic shifts
\[ \B^{-1}(0^*), \B^{-1}(1^*), \B^{-1}(0^*20^*), \B^{-1}(0^* 1^*), \B^{-1}(1^* 0^*) \]
and their finite unions. In particular, $X_\tau$ itself is countable. \qee
\end{example}

The following uncountable example is the basis of the more general construction in Lemma~\ref{lem:Construction}.

\begin{example}
\label{ex:UncountableSubstitution}
Let $\tau$ be the substitution $(0 \mapsto 00; 1 \mapsto 101)$. Then one can show by induction that $\tau^n(0) = 0^{2^n}$ and
\[ \tau^n(1) = \tau^{n-1}(1) 0^{2^n} \tau^{n-1}(1) = 1 0^{2^{\phi_0}} 1 0^{2^{\phi_1}} 1 0^{2^{\phi_2}} 1 \cdots 1 0^{2^{\phi_k}} 1 \] where $\phi$ is the ruler sequence and $k = 2^n-2$. Clearly, $\B^{-1}(0^*)$ is a subsystem. We claim that it is in fact the only subsystem.

Namely, we show that if $x \in X_\tau$ contains the symbol $1$, then it is weakly transitive. Suppose $x_0 = 1$ (by shifting $x$ if necessary). For any $k$, by definition of $X_\tau$ we must have that $x_{[-k,k]} \sqsubset \tau^n(1)$ for some $n$. We note that the ruler sequence has unique singularities, and thus does not contain subwords of the form $ab$ where $a, b \in \N$ and both are arbitrarily large -- in fact, every second symbol in it is $0$, and thus every second gap between two symbols $1$ is of length $1$. Taking a suitable $k$ and considering the word $\tau^n(1)$ such that $x_{[-k,k]} \sqsubset \tau^n(1)$, we see that either $x_{-2} = 1$ or $x_2 = 1$. For concreteness, suppose we are in the second case. Again, due to unique singularities, the ruler sequence does not contain words of the form $a0b$ where both $a$ and $b$ are arbitrarily large, and in fact one of them is always $1$. This means, again taking suitable $k$ and looking at the corresponding $\tau^n(1)$, that either $x_{-4} = 1$ or $x_{2+4} = 1$. For concreteness, suppose we are in the second case. The only extension of $01$ to the right by one symbol in the ruler sequence is $010$, so again looking at $\tau^n(1)$ for large $n$, we see that $x_8 = 1$, and $x_{[0,8]} = 101000101$.

We repeat this deduction infinitely many times: We look at the finite part of $x$ already considered by the process, and observe, using the unique singularities of the ruler sequence, that it cannot be the case that all the $1$s in $x$ are in this part, and in fact the already filled part must be continued by $0^{2^\ell}1$ to the right or $10^{2^\ell}$ to the left, for suitable $\ell$. We then extend this word deterministically using the properties of the ruler sequence given in Section~\ref{sec:Ruler}.

If we eventually fill the whole point $x$ with this deducion process (so that we find the new $1$ from both the left and the right side infinitely many times), then $x$ corresponds in a one-to-one fashion to a point $\psi \in \N^\Z$ in the orbit closure of the ruler sequence. If only one side is filled, then one can check that $x$ is in the orbit closure of either \[^\infty 0 . 1 0^{2^{\phi_0}} 1 0^{2^{\phi_1}} 1 0^{2^{\phi_2}} 1 0^{\phi_3} 1 \ldots \]
or
\[ \ldots 1 0^{2^{\phi_3}} 1 0^{2^{\phi_2}} 1 0^{2^{\phi_1}} 1 0^{2^{\phi_0}} 1 . 0^\infty, \]
and in either case it generates the whole subshift $X_\tau$. \qee
\end{example}

We believe that all substitutions generate quasiminimal systems. In the following, we show this under an additional condition: Let $\tau : S \to S^+$ be a substitution and let $S_\ell = \{a \in S \;|\; |\tau^n(a)| \rightarrow \infty\}$, the set of \emph{long symbols}. If for some $m$,
\[ w \sqsubset X_\tau \wedge |w| \geq m \implies S_\ell \cap \B_1(w) \neq \emptyset, \]
then we say \emph{long symbols are syndetic in $X_\tau$}.

Proposition~5.5 of \cite{BeKwMe09} shows that the above condition holds automatically if $X_\tau$ contains no periodic points, and in Proposition~5.6 they show that if this condition holds, then the number of minimal subshifts contained in $X_\tau$ is at most $|S|$. We show a similar result for the set of all subshifts.

\begin{proposition}
\label{prop:SubstitutionsQuasiminimal}
If $\tau : S \to S^+$ is a substitution and long symbols are syndetic in $X_\tau$, then $X_\tau$ is quasiminimal.
\end{proposition}

\begin{proof}
For each $n \geq 1$ and $x \in X_\tau$ there exists $y \in X_\tau$ such that $x = \tau^n(y)$ (up to shifting $x$). Namely, each subword of $x$ occurs in $\tau^k(a) = \tau^n(\tau^{k-n}(a))$ for arbitrarily large $k$ and a long symbol $a$, and we obtain $y$ as a limit point of $\tau^{k-n}(a)$. Let $m$ be such that every word of length $m$ in $X_\tau$ contains a long symbol. Let $Z \subset X_\tau$ be a subshift, and for each $n \in \N$, associate to $Z$ the following set:
\[ W_n(Z) = \{ w \in S^{m+1} \;|\; \exists x, y \in X_\tau: x \in Z \wedge x = \tau^n(y) \wedge w \sqsubset y \}. \]

Note that for each $n$, $W_n$ can have at most $2^{|S|^{m+1}}$ distinct values. Thus, if $Z_1, Z_2, \ldots, Z_{2^{|S|^{m+1}}+1}$ are subsystems of $X_\tau$, we must have two indices $i \neq j$ such that $W_n(Z_i) = W_n(Z_j)$ for infinitely many $n$. It is enough to show that this implies $\B(Z_i) = \B(Z_j)$, since subshifts with the same language are equal.

Thus, suppose that $Z, Y$ are subshifts of $X_\tau$ with $W_n(Z) = W_n(Y)$ for arbitrarily large $n$. Let $u \sqsubset Z$ be arbitrary. Let $k$ be such that $|\tau^k(a)| \geq |u|$ whenever $a$ is a long symbol, and choose $n \geq k$ such that $W_n(Z) = W_n(Y)$. Since $\tau(S) \subset S^+$, $\tau$ can only increase the length of words, and thus we have $|\tau^n(a)| \geq |u|$ for all long symbols $a$.

Choose a point $z \in Z$ such that $u \sqsubset z$, and let $z = \tau^n(y)$ where $y \in X_\tau$. Since long symbols are syndetic in $X_\tau$ and $|\tau^n(a)| \geq |u|$ for long symbols $a$, it is easy to see that there is a subword $w$ of $y$ of length $m+1$ with $u \sqsubset \tau^n(w)$. By the assumption $W_n(Z) = W_n(Y)$, $w$ is also a subword of some $y' \in X_\tau$ such that $\tau^n(y) \in Y$, so $u \sqsubset Y$.

This shows $u \sqsubset Y$, and since $u \sqsubset Z$ was arbitrary, we have $\B(Z) \subset \B(Y)$. Symmetrically, we obtain $\B(Y) \subset \B(Z)$, which concludes the proof. 
\end{proof}

The upper bound we obtain for the number of subsystems $2^{S^{m+1}}$, where $m$ is the bound for the length of the gap between two long symbols. In the case that all symbols are long, we have $m = 1$, and obtain the upper bound $2^{|S|^2}$ for the number of subsystems. The next example shows that this is in the right ballpark.\footnote{We only look at the case $m = 1$ for simplicity, but one can add short symbols between long symbols in the word $\tau(a) = w$ in the example, to get roughly $2^{|S_\ell|^2 |S \setminus S_\ell|^{m-1}}$ subshifts, for a partition $S = S_\ell \cup (S \setminus S_\ell)$.}

\begin{example}
Let $k \in \N$, let $S = \{a, b_1, \ldots, b_k\}$, and define the substitution $\tau$ with $\tau(b_i) = b_i^2$ for all $i \in [1,k]$ and $\tau(a) = w$, where $w \in \{b_1, \ldots, b_k\}$ is such that $b_i b_j \sqsubset w$ for all $i, j \in [1, k]$. It is easy to see that $X_\tau = \B^{-1}(\bigcup_{i,j} b_i^* b_j^*)$.

Choose a subset $K \subset [1, k]$, and a subset
\[ J \subset \{ (i, j) \in [1, k]^2 \;|\; i \neq j \} \]
such that $(i, j) \in J \implies \{i, j\} \subset K$. For such $(K, J)$, let
\[ Y_{K, J} = \B^{-1}(\bigcup_{i \in K} b_i^* \cup \bigcup_{(i, j) \in J} b_i^* b_j^*). \]

We note that we can determine the sets $K$ and $J$ from $Y_{K,J}$, and if $Y$ is a subshift of $X_\tau$, then $Y = Y_{K,J}$ for some $(K, J)$ with the above properties. Thus, to compute the number $B(k)$ of subsystems of $X_\tau$, we only need to compute the number of pairs $(K, J)$ with these properties. The number of such pairs is just the number of directed graphs whose vertices form a subset of $[1, k]$.

Letting $A(j) = 2^{j(j-1)}$ be the number of directed graphs with vertices $[1,j]$, we have
\[ B(k) = \binom{k}{0} A(0) + \binom{k}{1} A(1) + \binom{k}{2} A(2) + \cdots + \binom{k}{k} A(k). \]
Of course, $B(k) \geq A(k) = 2^{j(j-1)}$, so there exists a substitution on an alphabet of $k$ symbols with $B(k-1) \geq 2^{(k-1)(k-2)}$ subsystems. \qee
\end{example} 

In the OEIS database, $B(i)$ is the sequence A135756, and the first few values are
\[ B(0) = 1, B(1) = 2, B(2) = 7, B(3) = 80, B(4) = 4381, B(5) = 1069742. \]
We note that the upper bound for the number of subshifts of $X_\tau$ which are unions of minimal subshifts is $2^{|S|}$ by the result of \cite{BeKwMe09}. For $|S| = 6$, the number of such subsystems of $X_\tau$ is at most $2^6 = 64$, but one can have $B(5) = 1069742$ subsystems in total by the previous example.

Example~\ref{ex:CountableSubstitution2} and Example~\ref{ex:UncountableSubstitution} both have the property of syndetic long symbols, as all symbols are long. Example~\ref{ex:CountableSubstitution1} does not have this property, as $0$ is not a long symbol, but clearly the subshift does not change if the image of $0$ is changed to $00$. However, we believe there is no substitution with syndetic long symbols generating a subshift conjugate to $X_\tau$ where
\[ \tau = (0 \mapsto 0; 1 \mapsto 10; 2 \mapsto 021), \]
where the point generated by the symbol $2$ is
\[ ...0000002110100100010000100000... \]
Of course, even though long symbols are not syndetic, it is not very hard to show that this subshift is quasiminimal. In general, one might be able to obtain a proof for general substitutions by analysing the proof of Proposition~5.5 in \cite{BeKwMe09} in more detail.


We show that the model-checking problem for regular languages is decidable for these systems.

\begin{lemma}
\label{lem:RegularIntersection}
Given a substitution $\tau : S \to S^*$, a symbol $a \in S$ and a regular language $L$, it is decidable whether $\tau^n(a) \in L$ for some $n \in \N$. In particular, it is decidable whether $\B(X_\tau) \cap L = \emptyset$ for a given regular $L$.
\end{lemma}

\begin{proof}
We may assume $L \subset S^*$. Let $A$ be a nondeterministic finite state automaton for $L$ with state set $Q$, initial state $q_s$, final state $q_t$ and transition function $\delta \subset Q \times S \times Q$, and extend $\delta$ to a relation $\delta \subset Q \times S^n \times Q$ for all $n$ in the usual way:
\[ (q, w, q') \in \delta \iff (q, w_0, q_0), (q_0, w_1, q_1), \cdots, (q_{n-2}, w_{n-1}, q') \in \delta \]
For all $s \in S$, let $R \in ((2^{Q \times Q})^S)^\N$ be defined by
\[ (q, q') \in (R_i)_s \iff \delta(q, \tau^i(s), q'). \]

We can compute $R_i$ for each $i$ easily from the definition of $\delta$, and $R_i$ takes its value in the finite set $(2^{Q \times Q})^S$. The set $(2^{Q \times Q})^S$ is finite, so let $t, p$ be such that $R_t = R_{t+p}$. Then, $R_t' = R_{t'+p}$ for all $t' > t$ as well: writing $w = \tau^{t' - t}(s)$, we have
\begin{align*}
(q, q') \in (R_{t'})_s &\iff \delta(q, \tau^{t'}(s), q') \\
&\iff \delta(q, \tau^t(w), q') \\
&\iff \delta(q, \tau^{t+p}(w), q') \\
&\iff \delta(q, \tau^{t'+p}(s), q') \\
&\iff (q, q') \in (R_{t'+p})_s,
\end{align*}
where the second $\iff$ follows because $R_t = R_{t+p}$, so that
\[ \delta(q, \tau^t(a), q') \iff \delta(q, \tau^{t+p}(a), q') \]
for all $a \in S$, and the third $\iff$ because this extends to all words $w \in S^*$ by the inductive definition of $\delta \subset Q \times S^n \times Q$.

Now, to check whether $\tau^n(a) \in L$ for some $n$, it is enough to check whether $(q_s, q_t) \in (R_i)_a$ for some $i \in [0, p+t-1]$. The second claim is proved by checking whether we have $\tau^n(a) \in S^*LS^*$ for some $n$ and $a \in S$.
\end{proof}

\subsection{Universal non-quasiminimal examples}

We present our (non-quasiminimal) example of a universal recursive subshift with finitely many minimal subsystems.

\begin{proposition}
\label{prop:OneMinimal}
There exists a recursive subshift $X \subset \{0,1,2,3\}^\Z$ which is contained in a countable SFT, has finitely many minimal subsystems, and has a $\Sigma^0_1$-complete halting problem. Every minimal subsystem of $X$ consists of a single unary point.
\end{proposition}

\begin{proof}
Enumerate the deterministic Turing machines as $T_1, T_2, T_3, \ldots$. We define a subshift by $X = \overline{\mathcal{O}(\{x_i \;|\; i \geq 1\})}$, where
\[ x_i = \INF 0.1^i2^{i + h(i)}3\INF, \]
if the machine $T_i$ halts exactly after $h(i) \in \N$ steps, and $x_i = \INF 0.1^i2\INF$ if it never halts.

This is a recursive subshift: First, all words $a^*b^*$ for $a \leq b \in \{0,1,2,3\}$ are in $\B(X)$. Let $i, j, k, \ell \geq 1$ be arbitrary. The word $0^i 1^j 2^k$ is in $\B(X)$ if and only if $T_j$ does not halt in the first $k$ steps, which is decidable. The word $0^i 1^j 2^k 3^\ell$ is in $\B(X)$ if and only if $T_j$ halts exactly after $k-j$ steps, which is decidable. The word $1^j 2^k 3^\ell$ is in $\B(X)$ if and only if some word of the form $0 1^{j'} 2^k 3^\ell$ is in $\B(X)$, where $j \leq j' \leq k$, which we can check by running all the machines $T_{j'}$ for at most $k$ steps. 

The subshift $X$ is contained in the countable SFT $\B^{-1}(0^*1^*2^*3^*)$, because each of the points $x_i$ is in this subshift. The only periodic points in $X$ are the points $\INF a \INF$ for $a \in \{0,1,2,3\}$. Since a minimal subshift is either finite or uncountable, the only minimal subsystems of $X$ are these singleton subshifts.

The undirected halting problem of $X$ is $\Sigma^0_1$-complete because solving the halting problem for the clopen sets $[01^i2]_0$ and $[3]_0$ is equivalent to solving the halting problem of $T_i$. 
\end{proof}

We make a few remarks about this subshift, omitting the easy proofs. 

\begin{enumerate}
\item This subshift has Cantor-Bendixson rank $3$. Countable subshifts of Cantor-Bendixson rank $1$ are finite, and those of rank $2$ are sofic \cite{SaTo14a}, so that this is the minimal possible rank for a countable universal subshift.

\item By forbidding the single letter $3$ from the subshift $X$, we obtain a system $Y$ whose language is not decidable, since $0 1^k 2 \sqsubset Y$ if and only if $T_k$ never halts. Thus, while $X$ is recursive, one could say it is not \emph{hereditarily recursive}, as it contains a $\Pi^0_1$ subshift which is not recursive. (Compare this with Corollary~\ref{cor:SubRecursive}.)

\item The subshift $X$ has only finitely many minimal subshifts, but more than one. By using the points
\[ x_i = \INF 0.10^{i}20^{i + h(i)}30\INF \]
in the proof, the enveloping countable SFT of Cantor-Bendixson rank~$4$ changes into a countable sofic shift with the same CB-rank, and there will be only one minimal subsystem.

\item One could say that we are cheating, and that a finite subsystem is not a sufficiently interesting example of a minimal subshift; indeed, in many contexts it makes sense to not even call such systems minimal, to avoid having to discuss these trivial cases. If, in the points $x_i = \INF 0.10^{i}20^{i + h(i)}30\INF$, we replace the maximal subwords of the form $0^k$ by (any!) subwords of length $k$ of a minimal subshift $Y$ over an alphabet disjoint with $\{1,2,3\}$, it is easy to check that the only minimal subsystem of $\overline{\mathcal{O}(\{x_i \;|\; i \geq 1\})}$ is~$Y$. Choosing the minimal subshift and the words suitably, the subshift can be made recursive as well.

\item A classical tool for studying minimal systems are the Bratteli-Vershik systems. It was proved in \cite{HePuSk92} that every minimal system is conjugate to such a system. In fact, the result of \cite{HePuSk92} applies more generally to systems containing exactly one minimal subshift, and thus to our example. Such systems are called \emph{essentially minimal systems}. The classes of essentially minimal systems and quasiminimal systems are incomparable.

\item In Proposition~9 of \cite{DeKuBl06}, it is shown that if the limit set of a symbolic dynamical system is a finite union of minimal systems, then it is decidable. As we are concerned with two-way subshifts, the limit set of a subshift is equal to the subshift itself. We note, however, that the \emph{asymptotic set} (the union of limit points of individual configurations) and its \emph{nonwandering set} (the points whose neighborhoods all return to themselves) of $X$ are both finite unions of minimal systems: in fact, these sets are finite, and can again be taken to be singletons.
\end{enumerate}

As discussed in the introduction, our system is a $\Z$-system, and one could also ask whether Conjecture~\ref{con:IntroManySubsystems} is true for $\N$-systems, where a priori there are more subsystems. Since minimal systems are clearly surjective, the following proposition shows that Proposition~\ref{prop:OneMinimal} resolves the case of $\N$-actions too. If $X \subset S^\N$ is a one-sided subshift, we say it is \emph{surjective} if the left shift $\sigma : X \to X$ is surjective.

\begin{proposition}
\label{prop:Correspondence}
For any subshift $X \subset \N^\Z$ let
\[ c(X) = \{y \in \N^\N \;|\; \exists x \in \N^{-\N}: x.y \in X\}. \]
For any surjective subshift $X \subset \N^\N$, let
\[ e(X) = \{x \in \N^\Z \;|\; \forall i \in \Z: x_{[i, \infty)} \in X\}. \]
These operations preserve the language of the subshift, and thus $c(e(X)) = X$ for any surjective $\N$-subshift, and $e(c(X)) = X$ for any $\Z$-subshift.
\end{proposition}

The operations $c$ and $e$ and their correspondence are well-known, although we do not know an explicit reference for the precise statement above.

\subsection{Uncountable universal quasiminimal examples}

We now move on to our quasiminimal examples in the case of $\Z$-actions, from which one can obtain results in the case of $\N$-actions from the previous proposition, when only surjective subsystems are considered. The case of $\N$-actions and non-surjective subsystems is dealt with in Section~\ref{sec:NActions}. First, we give our example of a universal (uncountable) quasiminimal subshifts, giving our first proof of Theorem~\ref{thm:IntroQuasiminimal}. In this case, we can construct a subshift whose halting problem along a clopen set is undecidable, and thus so is the model-checking problem for local languages, piecewise testable languages, starfree languages and renewal languages. Our example has only one proper subshift, which can be chosen rather freely.

\begin{theorem}
\label{thm:TransitiveLocallyTestableLanguages}
For any recursive infinite minimal subshift $Y$, there exists a transitive uncountable quasiminimal subshift $X$ for which the halting problem along a clopen set is $\Sigma^0_1$-complete, such that the only nontrivial subsystem of $X$ is $Y$.
\end{theorem}

To prove Theorem~\ref{thm:TransitiveLocallyTestableLanguages}, we need a few lemmas. The first is a method of constructing quasiminimal subshifts.

\begin{lemma}
\label{lem:Construction}
Let $Y \subset S^\Z$ be a quasiminimal subshift. Let $\phi \in \N^\Z$ be uniformly recurrent with unique singularities, and let $\tau : \N \to a\B(Y)$ be a substitution where $|\tau(n)| \overset{n \rightarrow \infty}{\longrightarrow} \infty$, where $a$ is a symbol not in $S$. Then $X = \OC{\tau(\phi)} \subset (S \cup \{a\})^\Z$ is transitive and quasiminimal, and in fact every proper subsystem of $X$ is a subsystem of $Y$.
\end{lemma}

\begin{proof}
If $Z$ is a subshift of $X$ where $a$ does not occur, then clearly it is also a subshift of $Y$. Thus, we only need to show that there are finitely many subshifts of $X$ where $a$ does occur. In fact, we show the stronger fact that $a \in x \implies \OC{x} = X$. For this, suppose that $a \in x$.

If $a$ occurs infinitely many times in both tails of $x$ (that is, $\{i \in \Z \;|\; x_i = a\}$ is unbounded from both above and below), then by $|\tau(n)| \overset{n \rightarrow \infty}{\longrightarrow} \infty$ and a compactness argument, we have $x \in \mathcal{O}(\tau(\phi'))$ for some $\phi' \in \OC{\phi} \cap \N^\Z$. Namely, since $x \in \OC{\tau(\phi)}$, for any $j, j'$ such that $x_j = x_{j'} = a$, we find $i, i'$ such that $x_{[j, j']} = \tau(\phi_{[i,i']})a$. There are finitely many choices for the word $\phi_{[i,i']}$ for each pair $j, j'$, so letting $j \rightarrow -\infty$ and $j' \rightarrow \infty$ and passing to a suitable subsequence, we obtain $x \in \mathcal{O}(\tau(\phi'))$ for some $\phi' \in \OC{\phi}$.

Because $\phi$ is uniformly recurrent, $\phi'$ contains all its finite patterns, and thus $\OC{x}$ contains all finite patterns of $X$, which implies $\OC{x} = X$.

Next, suppose $a$ occurs infinitely many times in one tail of $x$, but not the other. These cases are (more or less) symmetric, so we suppose $x_0 = a$, $x_i \neq a$ for all $i < 0$, and $\{i \in \N \;|\; x_i = a\}$ is unbounded from above. Now, a compactness argument like the one above shows that there exists a point $\phi' \in \OC{\phi}$ such that $\phi'_{-1} = \infty$, $x_{\N} = \tau(\phi'_{\N})$, and $x_{(-\infty,-1]}$ is some limit point of the suffices of words $\tau(n)$. Because $\phi$ uniformly recurrent, all its finite words appear in the one-way point $\phi'_{\N}$, and thus also in $x$. Again, $\OC{x} = X$.

Finally, we show that $a$ cannot occur finitely many times. Namely, suppose $x_j = a, x_\ell = a$, and $x_i \neq a$ for all $i \notin [j, \ell]$. Then a compactness argument shows that there is a point $\phi' \in \OC{\phi}$ with $|\phi'|_\infty \geq 2$. This is a contradiction, since $\phi$ was assumed to have unique singularities.

The transitivity of $X$ is clear from the definition. 
\end{proof}

\begin{definition}
We say a subshift $X \subset S^\Z$ is \emph{right-perfect} if the subshift $c(X) = \{y \;|\; \exists x: x.y \in X\}$ is perfect.
\end{definition}

In other words, $X$ is right-perfect if every word of $X$ has at least two incomparable extensions to the right.

\begin{lemma}
\label{lem:PrefixCode}
Let $Y \in [0, k-1]^\Z$ be a nonempty right-perfect subshift with a recursive language. Then, there exists a recursively enumerable infinite prefix code $u_1, u_2, \ldots$ of words in $Y$ such that $(u_j)_{[0,|u_j|-2]} = (u_{j+k})_{[0,|u_j|-2]}$ for all $j, k \geq 1$.
\end{lemma}

Of course, symmetrically, there exists such a suffix code if $Y$ is left-perfect.

\begin{proof}
Note that a right-perfect nonempty subshift is infinite (even uncountable). Enumerate the words of $Y$ as $V_1 = v_1, v_2, \ldots$, first ordered by length, and then lexicographically among words of each length. Let $u_1 = ua$ be the first one-symbol extension of a word of $Y$ on this list which has at least two one-symbol extensions, $ua$ and $ub$. Let $V_2$ be the subsequence of $V_1$ of words beginning with $ub$. Having chosen $u_1, \ldots, u_j$ and restricted our list to $V_{j+1}$, choose again the lexicographically smallest one-symbol extension of a word of $Y$ which has at least two one-symbol extensions in $V_{j+1}$, and restrict to $V_{j+2}$ accordingly. Since $Y$ is right-perfect, this process continues forever, and the resulting set of words is clearly a prefix code. The condition on compatible prefixes is automatic in the construction. 
\end{proof}

\begin{proof}[Proof of Theorem~\ref{thm:TransitiveLocallyTestableLanguages}]
Suppose $Y \subset [1, k]^\Z$. The subshift $X$ will be over the alphabet $[0, k]$.

Enumerate the deterministic Turing machines as $T_0, T_1, T_2, \ldots$. Let $\phi \in \N^\Z$ be a computable point in the orbit closure of the ruler sequence.

Using Lemma~\ref{lem:PrefixCode} and the fact that an infinite minimal subshift is left- and right-perfect, take recursively enumerable prefix and suffix codes of words $u_i$ and $v_i$ of $Y$, respectively. Let $h : \N \to \N$ be a computable infinite-to-one dovetailing of the natural numbers (for example, the ruler sequence). Now, let $x = \tau(\phi)$, where $\tau$ is the substitution
\[ i \mapsto \begin{array}{ll}
0 u_{h(i)} w_i v_{h(i)}, & \mbox{if $T_{h(i)}$ does not halt before step $i$.} \\
0 u_{h(i)} w_i v_{h(i)+1}, & \mbox{if $T_{h(i)}$ halts before step $i$.} \\
\end{array} \]
where each $w_i$ is a word of length at least $i$ such that $u_{h(i)} w_i v_{h(i)} \sqsubset Y$, chosen in such a way that $i \mapsto w_i$ is computable, and as $i$ runs over the natural numbers, all words of $Y$ beginning with $u_{h(i)}$ appear as prefixes of words $u_{h(i)} w_i$ infinitely many times (and similarly for $w_i v_{h(i)}$). Let $X = \overline{\mathcal{O}(x)} \subset [0, k]^\Z$.

Clearly, $X$ is transitive, and $Y$ is its subsystem. By Lemma~\ref{lem:Construction} it is quasiminimal and all nontrivial proper subsystems are subsystems of $Y$ -- thus equal to $Y$ by minimality. The halting problem along a clopen set is undecidable for $X$ because $T_j$ eventually halts if and only if $\B(X) \cap L \neq \emptyset$, where $L = 0u_j [1,k]^* v_{j+1}0$.

To show that the subshift is recursive, note that given any word $0u0$ where $u \in [1,k]^*$, we can easily check whether there exists $n \in \N$ such that $\tau(n) = 0u$. Suppose then that we are given a word $t_0 0 t_1 0 t_2 0 \cdots 0 t_k$, where $t_i \in [1,k]^*$. Using the properties of the sequence $\phi$, we can compute two extensions of this word, one beginning with $0$ and one ending in $0$, such that every extension agrees with one of them. Thus, we may assume the given word is $w = 0 t_0 0 t_1 0 t_2 0 \cdots 0 t_k$.

Such a word is in the language of $X$ if and only if $u = \tau^{-1}(0 t_0 0 t_1 0 t_2 0 \cdots 0 t_{k-1})$ is well-defined and $u \sqsubset \phi$ holds, and either $t_k$ begins with one of the words $u_{h(i)}$ or it is a prefix of the unique one-way limit $x \in [1, k]^\N$ of the words $u_i$. These conditions are easily seen to be decidable. 
\end{proof}

\subsection{Countable universal quasiminimal examples}

The case of countable quasiminimal subshifts is also interesting. For such subshifts, we show that both the modular halting problem and the counting problem are undecidable, so that the model-checking problems for piecewise testable languages, starfree languages and renewal languages are undecidable as well. In Section~\ref{sec:Decidability}, we complement these results by showing that in each, the number of subsystems is optimal, and the model-checking problem for local languages (and thus the halting problem along a clopen set) is decidable.

First, we show that if nonzero symbols are asymptotically spaced far apart in $x$, then $x$ generates a quasiminimal countable subshift.

\begin{lemma}
\label{lem:CountableConstruction}
Let $x \in (\{0\} \cup S)^\Z$ be such that
\[ \forall n: \exists m: (x_j = x_{j'} \neq 0 \wedge j \neq j' \wedge |j| \geq m \implies |j - j'| > n). \]
Then $X = \OC{x}$ is a countable weakly transitive quasiminimal subshift whose proper subshifts are among the subshifts of $\B^{-1}(0^*S0^*)$. If $x$ is computable and $m = m(n)$ is computable in $n$, then $X$ is recursive.
\end{lemma}

\begin{proof}
Every limit point of $x$ is easily seen to be in $\B^{-1}(0^*S0^*)$, so $X$ is contained in $\mathcal{O}(x) \cup \B^{-1}(0^*S0^*)$. Thus, $X$ is countable and quasiminimal, and its proper subshifts are among the subshifts of $\B^{-1}(0^*S0^*)$.

Suppose then that $x$ and $m$ are computable. For each ${^\infty 0} s 0^\infty \notin X$, let $k_s$ be the maximal $|i|$ such that $x_i = s$ (given  to the algorithm by a look-up table). Given $w$, if $w$ is not a word in $\B^{-1}(0^*S0^*)$, then $w$ contains two nonzero symbols spaced $n$ apart, so $w \sqsubset X$ if and only if $w$ occurs in $x_{[-m(n)-|w|, m(n)+|w|]}$. If $w \in 0^*$, or $w \in 0^*s0^*$ and ${^\infty 0} s 0^\infty \in X$, then $w \sqsubset X$. If $w \in 0^*s0^*$ and ${^\infty 0} s 0^\infty \notin X$, then $w \sqsubset X$ if and only if $w \sqsubset x_{[-k_s-|w|, k_s+|w|]}$. 
\end{proof}

We begin with the case of modular halting problem. We first show the result for the modular halting problem along a clopen set, as the proof illustrates the main idea, but is easier.

\begin{proposition}
\label{prop:CountableCodedLanguages}
There exists a recursive countable weakly transitive quasiminimal subshift $X \subset \{0,1\}^\Z$ for which the modular halting problem along a clopen set is $\Sigma^0_1$-complete, and which has exactly two subsystems $\B^{-1}(0^*10^*)$ and $\B^{-1}(0^*)$.
\end{proposition}


\begin{proof}
Enumerate the deterministic Turing machines as $T_0, T_1, T_2, \ldots$. Let $h : \N \to \N$ be the ruler sequence and let $p_i$ denote the $i$th odd prime number (so $p_1 = 3$). Let $\tau$ be the substitution
\[ i \mapsto \begin{array}{ll}
1 0^{2^i}, & \mbox{if $T_{h(i)}$ does not halt before step $i$.} \\
1 0^{p_{h(i)} 2^i}, & \mbox{if $T_{h(i)}$ halts before step $i$.} \\
\end{array} \]
We let $x = {^\infty} 0.\tau(0123...)$ and $X = \OC{x}$.

The modular halting problem of this subshift is clearly undecidable, as the distance of two symbols $1$ along symbols $0$ can be divisible by $p_j$ if and only if $T_j$ eventually halts. The required properties of $X$ follow from Lemma~\ref{lem:CountableConstruction}. 
\end{proof}

\begin{theorem}
\label{thm:CountableCodedLanguages}
There exists a recursive countable weakly transitive quasiminimal subshift $X \subset \{0,1\}^\Z$ for which the modular halting problem is $\Sigma^0_1$-complete, and which has exactly two subsystems $\B^{-1}(0^*10^*)$ and $\B^{-1}(0^*)$.
\end{theorem}

\begin{proof}
Enumerate the deterministic Turing machines as $T_0, T_1, T_2, \ldots$. Let $h : \N \to \N$ be the ruler sequence. Let $p(i)$ denote the $i$th odd prime number, and define the \emph{primorial} of $n \geq 2$ to be $n\# = 2 p(1) p(2) \cdots p(\ell)$ where $\ell$ is maximal such that $p(\ell) \leq n$. Thus, $n\#$ is the product of primes up to $n$. Let $f : \N \to \N$ be an increasing recursive function satisfying $p(f(i-1))^{p(f(i-1))} \leq p(f(i))$ for all $i$. Let $\tau$ be the substitution
\[ i \mapsto \begin{array}{ll}
1 0^{p(f(i))\#-1}, & \mbox{if $T_{h(i)}$ does not halt before step $i$.} \\
1 0^{p(f(i))\# + p(f(i))\#/p(f(h(i))) \cdot k_i - 1}, & \mbox{if $T_{h(i)}$ halts before step $i$,} \\
\end{array} \]
where $0 < k_i < p(f(h(i)))$ is minimal such that
\[ p(f(i))\#/p(f(h(i))) \cdot k_i \equiv 1 \bmod p(f(h(i))). \]
Note that such $k_i$ exists because $p(f(i))\#/p(f(h(i)))$ is not divisible by $p(f(h(i)))$.

We let $x = {^\infty} 0.\tau(0123...)$ and $X = \OC{x}$. 
To show that the modular halting problem is $\Sigma^0_1$-complete, we show that there exist two symbols $1$ with distance $\ell \equiv 1 \bmod p(f(j))$ if and only if $T_j$ eventually halts. First, if $T_j$ does halt, then $j = h(i)$ for some $i$ such that $T_j$ halts before step $i$. Then,
\[ 1 0^{p(f(i))\# + p(f(i))\#/p(f(j)) \cdot k_i - 1}1 \sqsubset x, \]
where $\ell = p(f(i))\# + p(f(i))\#/p(f(j)) \cdot k_i \equiv 1 \bmod p(f(j))$ by the choice of $k_i$ and because $j \leq i$.

Let us show that if $T_j$ never halts, then no such distance $\ell$ occurs. First, note that if $T_j$ does not halt, then the distance between the $i$th and $(i+1)$th symbol $1$ produced by the construction is divisible by $p(f(j))$ by construction whenever $i \geq j$. Thus, if there is a distance $\ell \equiv 1 \bmod p(f(j))$ between two symbols $1$ in $x$, then it is among the first $j$ symbols, that is, the distance must occur between two symbols $1$ in the word $\tau(1 2 \cdots (j-1)) \cdot 1$. The distance is never $1$, so it must be at least $p(f(j)) + 1$. However, we have
\[ |\tau(1 \cdots (j-1)) \cdot 1| \leq 2j|\tau(j-1)| = 2jp(f(j-1))\# \leq p(f(j-1))^{p(f(j-1))} \leq p(f(j)), \]
by the assumption on $f$.

This shows that the modular halting problem is undecidable. The other properties again follow from Lemma~\ref{lem:CountableConstruction}. 
\end{proof}

Of course, the growth rate for the distances between consecutive symbols $1$ is not optimal, but we are not aware of essentially simpler substitutions for which the proof is equally short.

We note that just like we could take an arbitrary infinite recursive minimal subshift as the unique minimal subshift in the proof of Theorem~\ref{thm:TransitiveLocallyTestableLanguages}, one could of course use any infinite subshift of the form $\B^{-1}(u^* v u^*)$ in the place of $\B^{-1}(0^*10^*)$, although small additional complications arise if $|u|\! \not| \; |v|$.

Next, let us consider the counting problem. For this, we need four subsystems.

\begin{theorem}
\label{thm:CountablePiecewiseTestableLanguages}
There exists a recursive countable weakly transitive quasiminimal subshift $X$ for which the counting problem is $\Sigma^0_1$-complete, and which has exactly four subsystems, $\B^{-1}(0^*(1 + 2)0^*)$, $\B^{-1}(0^*20^*)$, $\B^{-1}(0^*10^*)$ and $\B^{-1}(0^*)$.
\end{theorem}

\begin{proof}
Let $T_i$ be an enumeration of Turing machines and $h$ and infinite-to-one computable mapping. Let $\tau$ be the substitution
\[ i \mapsto \begin{array}{ll}
2 0^i, & \mbox{if $T_{h(i)}$ does not halt before step $i$.} \\
2 (0^i 1)^{h(i)} 0^i, & \mbox{if $T_{h(i)}$ halts before step $i$,} \\
\end{array} \]
We let $x = {^\infty} 0.\tau(0123...)$ and $X = \OC{x}$. The counting problem is clearly $\Sigma^0_1$-complete, because there exists a point that travels from the cylinder $[2]$ back to itself along $[0]$ visiting $[1]$ exactly $j$ times if and only if $T_j$ eventually halts. The other properties again follow from Lemma~\ref{lem:CountableConstruction}. 
\end{proof}

\subsection{Model-checking for context-free languages}

In \cite{DeKuBl06}, it is shown that the model-checking problem of regular languages is decidable for minimal systems, and in this section, we have shown that this problem is undecidable for more complex subshifts. In another direction, we could ask how far we must step from the class of regular languages to find undecidable model-checking problems.

We show that the model-checking problem is hard at least for context-free languages. Our example is a subshift of the \emph{Dyck shift (with labels $\{1,2,3\}$)}, the subshift of $S = \{[_1, [_2, [_3, ]_3, ]_2, ]_1\}^\Z$ where the parentheses are balanced, in the sense that the process of recursively erasing subwords of the forms $[_s]_s$ for $s \in \{1,2,3\}$ never introduces a subword of the form $[_s]_{s'}$ with $s \neq s'$.

The language of the Dyck shift is the set of factors of the \emph{Dyck language} $L$ generated by the context-free grammar $A \mapsto AA | [_aA]_a | [_bA]_b | [_cA]_c | \lambda$. A deterministic push-down automaton $M$ for it is obtained by pushing $s$ on input $[_s$, and popping the corresponding symbols on closing brackets (rejecting the word if these do not match). For two words $u, v$ over $S$ which occur as subwords of words in $L$, write $u \sim v$ if the two words correspond to the same element of the syntactic monoid of $L$. (In other words, they close the same parentheses, and leave the same parentheses open).

\begin{theorem}
\label{thm:ContextFree}
There exists a minimal $\Pi^0_1$ subshift $X$ whose model-checking problem for context-free languages is $\Sigma^0_1$-complete.
\end{theorem}

\begin{proof}
For $s \in \{1,2,3\}$ let ${[_i^0} = {[_i}$, and ${]_i^0} = {]_i}$. Let $W_0 = \{[^0_1, [^0_2, [^0_3, ]^0_3, ]^0_2, ]^0_1\}$. Inductively, suppose $W_i = \{[^i_1, [^i_2, [^i_3, ]^i_3, ]^i_2, ]^i_1\}$ is defined, and each word in $W_i$ is a concatenation of words of $W_{i-1}$, and contains all concatenations $uv$ of words $u, v \in W_{i-1}$ such that $uv$ is a subword of a word of $L$. We also inductively suppose that each $w \in W_i$ corresponds to one of the symbols $a \in S$ of the original alphabet in the sense that the push/pop action corresponding to $a$ is performed on $w$ when reading the very last symbol, and during the reading of proper prefixes, no existing data on the stack is popped, and all words put on the stack begin with $3$. More precisely, we suppose the words in $W_i$ are of the same length $n_i$, and for all $s \in \{1,2,3\}$, $([^i_s)_{[0,n_i-2]} \sim (]^i_s)_{[0,n_i-2]} \sim \lambda$, $([^i_s)_{n_i-1} = [_s$ and $(]^i_s)_{n_i-1} = {]_s}$, and if the stack of $M$ initially contains $v$, then after reading any proper prefix of a word $u \in W^i$, the stack contains either $v$ or $v3w$ for some word $w$.

We define $W_{i+1}$ as a concatenation of words of $W_i$, respecting the inductive assumptions. Let $T_1, T_2, T_3, \ldots$ be an enumeration of Turing machines, and let $h : \N \to \N$ be a computable infinite-to-one mapping. Let $u_i$ be any word with $u_i \sim \lambda$ containing all legal concatenations of pairs of words in $W_i$. For example,
\[ u_i = \prod_{s, s' \in \{1,2,3\}} [^i_s[^i_{s'}]^i_{s'}]^i_s [^i_{s'}]^i_{s'} \]
is such a word, no matter what order the pairs $s, s'$ are listed in.

Now, if $T_{h(i)}$ does not halt in $i$ steps or less, we define
\[ [^{i+1}_s = [^i_3 u_i ]^i_3 [^i_s \mbox{, and } {]^{i+1}_s} = [^i_3 u_i ]^i_3 ]^i_s. \]
If $T_{h(i)}$ does halt in $i$ steps or less, we define
\[ [^{i+1}_s = [^i_3 ([^i_1)^{h(i)} [^i_2 ]^i_2 (]^i_1)^{h(i)} u_i ]^i_3 [^i_s \mbox{, and } ]^{i+1}_s = [^i_3 ([^i_1)^{h(i)} [^i_2 ]^i_2 (]^i_1)^{h(i)} u_i ]^i_3 ]^i_s. \]

We let
\[ W_{i+1} = \{ [^{i+1}_1, [^{i+1}_2, [^{i+1}_3, ]^{i+1}_3, ]^{i+1}_2, ]^{i+1}_1 \}. \]
The subshift $X$ is defined as the set of limit points of (say) the words $[^j_1$ as $j \rightarrow \infty$.

Now, define $L_k \subset L$ as the language defined by the automaton $A_k$, which simulates the deterministic push-down automaton $A$ for the Dyck language, but additionally inspects the $k+2$ top symbols of the stack after every step for the word $3 1^k 2$. The automaton accepts the word $w$ if and only if $w \in L$, and $3 1^k 2$ was seen on the top of the stack during the execution.

We claim that $X$ is minimal, $\Pi^0_1$ and for $j$ sufficiently large, its language intersects $L_j$ if and only if $T_j$ eventually halts. For minimality, simply note that if $u \sqsubset X$, then $u \sqsubset [^i_1$ for some $i$ (by definition). This word appears in every word of $W_{i+1}$ and every point of $X$ is a concatenation of these words.\footnote{We note that for this argument, there is no need to have all legal pairs of symbols in $u_i$, only all symbols, although the definition of $X$ is more robust if all pairs occur.} To show $X$ is $\Pi^0_1$, we note that given any word $w$ we can check whether $w \in \B(X)$ by computing the sets $W_i$ until their lengths exceed that of $w$. Then by the choice of $u_i$, every word of length $|w|$ that appears in $X$ appears in $W_{i+1}$.\footnote{For this, on the other hand, having all pairs in $u_i$ is essential.}

Now, note that if $T_j$ eventually halts, then certainly $L_j$ intersects the language of $X$ because for any large enough $i$ with $h(i) = j$, $[^i_3 ([^i_1)^{h(i)} [^i_2 \sqsubset [^{i+1}_a$ and ${[^i_3 ([^i_1)^{h(i)} [^i_2} \sim [_3 ([_1)^{h(i)} [_2$, which pushes the word $31^j2$ on top of the stack of $M$.

Conversely, suppose $j$ is large. If there exists a subword of $X$ where $31^j2$ appears on top of the stack during the run of $A$ on the word, then in particular there exists such a word which is a subword of some $[^{i+1}_s$ or $]^{i+1}_s$. Choose the minimal such~$i$, and suppose that it occurs in $v = [^{i+1}_s$ (the case of a closing bracket being similar). Let $k$ be minimal such that $31^j2$ is on the top of the stack after reading $v_{[0,k]}$. If $v_{[0,k]}$ is not a concatenation of words of $W_i$, then the automaton has read a proper prefix of a word of $W_i$ up to coordinate $k$. By the inductive assumption, it has then written a word beginning with $3$ on the stack, or nothing. In the first case, the word $31^j2$ was written on top of the stack while reading that word of $W_i$, so $i$ is not minimal. In the second case, reading the prefix did not modify the stack, so $k$ is not minimal.

In the remaining case, we have read some concatenation of words of $W_i$. It is clear by the construction that $31^j2$ is written on top of the stack after reading some such prefix if and only if $T_j$ halts or this happened because for some prefix $[^i_3 u$ of $[^i_3 u_i ]^i_3 [^i_s$ wrote this on top of the stack. Since $u$ is a concatenation of at most $54$ words of $W_i$, $j > 54$ is enough to prevent this. 
\end{proof}

\section{Quasiminimality and decidability}
\label{sec:Decidability}

In this section, we give various decidability results, which complement the results of the previous section. We take a rather high-level approach to decidability, and usually only describe the logical deductions the algorithm must make to determine the right answer.

\subsection{Decidability in the general case}

One could say that it is a folklore result that a minimal $\Pi^0_1$ subshift has a recursive language; we are not aware of a reference that states and proves this explicitly, although the proof is given in multiple places. The result is shown in \cite{Ho09} in the case of multidimensional SFTs, and the proof works in general. It is also essentially shown in \cite{BaJe10}, but the connection is not made very explicit. We give a proof below.

\begin{theorem}
\label{thm:MinimalRecursive}
A minimal $\Pi^0_1$ subshift is recursive.
\end{theorem}

\begin{proof}
Given an alphabet $S$, a word $w \in S^*$ and a Turing machine enumerating the forbidden patterns of a nonempty minimal subshift $X$ over $S$, we show that it is decidable whether $w \sqsubset X$.

If $w \not\sqsubset X$, then an algorithm can eventually detect this by the assumption that $X$ is $\Pi^0_1$. If $w \sqsubset X$, then by uniform recurrence, there exists $n$ such that $w \sqsubset u$ for all $u \in \B_n(X)$. By compactness, there exists $k$ such that the SFT $Y$ defined by the first $k$ forbidden patterns $v_1, v_2, \ldots, v_k$ enumerated by the given Turing machine defining $X$ satisfies $\B_n(Y) = \B_n(X)$. It follows that after enumerating the first $k$ forbidden patterns, the algorithm has found an SFT $Y$ such that $X \subset Y$ and $w \sqsubset u$ for all $u \in \B_n(Y)$, and can deduce that $w \sqsubset X$ (on the assumption that $X$ is nonempty). 
\end{proof}

The algorithm is uniform in the given Turing machine and given word, though only when restricted to inputs that define nonempty subshifts. It is in fact easily seen to be undecidable whether a given Turing machine defines an empty subshift or a nonempty minimal subshift: the given Turing machine may output the forbidden patterns of any $\Pi^0_1$ minimal subshift as long as another simulated Turing machine does not halt, and forbid every symbol if it does halt.

We show that the non-uniform decidability result generalizes to the class of quasiminimal subshifts as well, in Theorem~\ref{thm:Recursive} below. As a quasiminimal subshift contains only finitely many subshifts, it is natural to approach such subshifts by induction on the number of proper subshifts they have, and we give the following notation for this.


\begin{definition}
For $X$ quasiminimal, define
\[ \QW(X) = |\{ Y \;|\; Y \subset X, Y \mbox{ is a subshift} \}|. \]
\end{definition}

The empty subshift and the subshift itself are also counted: if $X$ is nonempty and minimal, then $\QW(X) = 2$, and $\QW(\emptyset) = 1$.

We need the following basic structural results.

\begin{definition}
If $w \sqsubset X$ and $u \leq_X w$ for all $u \sqsubset X$, then $w$ is called a \emph{generator (for $X$)}.
\end{definition}

\begin{lemma}
\label{lem:Generator}
Let $X$ be quasiminimal. Then $X$ has a generator if and only if it is not the union of its proper subshifts.
\end{lemma}

\begin{proof}
First, we suppose that $X$ is the union of its proper subshifts. Then, whenever $w \sqsubset X$, there must exist a proper subshift $Y \subsetneq X$ with $w \sqsubset X$. If $Y \subsetneq X$, then $u \not\sqsubset Y$ for some $u \sqsubset X$, which implies $u \not\leq_X w$. Thus, no $w \sqsubset X$ is a generator for $X$.

Now, suppose $X$ is not the union of its proper subshifts. Let $Y \subsetneq X$ be the union of its proper subshifts. Note that since $X$ has only finitely many proper subshifts, $Y = Y_1 \cup \cdots \cup Y_k$ for some $k$ and subshifts $Y_i \subsetneq X$, and thus $Y$ is itself a proper subshift of $X$. Let $w$ be any word that occurs in $X$ but not $Y$. Then there exists no proper subshift $Z$ of $X$ containing $w$. It follows that, for any $u \sqsubset X$, the subshift of $X$ where $u$ is forbidden does not contain $w$, which means $u$ occurs in every point of $X$ containing $w$. It follows that $u \leq_X w$. 
\end{proof}

Note that if $w$ is a generator for $X$, then since $u \leq_X w$ for any $u \sqsubset X$, any point $x$ containing $w$ must contain every word $u \sqsubset X$, that is, every point $x \in X$ containing $w$ is weakly transitive (in particular, $X$ is weakly transitive).

\begin{lemma}
\label{lem:SubshiftsPi}
If $X$ is quasiminimal and $Y$ is a subshift of $X$, then there exists a finite list of words $v_1, v_2, \ldots, v_k$ such that $Y$ is obtained from $X$ by forbidding the words $v_i$. In particular, every subshift of a $\Pi^0_1$ quasiminimal subshift is $\Pi^0_1$.
\end{lemma}

\begin{proof}
Let $Y_i$ be the $i$th SFT approximation of $Y$, and consider the subshifts $Z_i = Y_i \cap X$. We have $\bigcap Y_i = Y$, and thus also $\bigcap_i Z_i = Y$. The $Z_i$ are a decreasing sequence of subshifts of $X$, and thus for some $j$ we have $Z_j = Z_{j+k}$ for all $k \geq 0$. It follows that $Y = Z_j$. Since $Y_i$ is an SFT, $Y_i \cap X$ is obtained from $X$ by forbidding a finite set of words. The latter claim is then trivial: the Turing machine outputting the forbidden words of a subshift $Y \subset X$ first outputs the finitely many words $v_i$, and then the forbidden words of $X$. 
\end{proof}

\begin{theorem}
\label{thm:Recursive}
A quasiminimal $\Pi^0_1$ subshift $X \subset S^\Z$ is recursive.
\end{theorem}

\begin{proof}
We proceed by induction on $\QW(X)$. First, if $\QW(X) = 1$, then $X = \emptyset$ and there is nothing to prove. Suppose then that $\QW(X) > 1$, and let $u \in S^*$. We have to decide whether $u \sqsubset X$. By the induction hypothesis and the previous lemma, for any proper subshift $Y \subsetneq X$ it is decidable whether $u \sqsubset Y$. If $u \sqsubset Y$ for some proper subshift, then the algorithm may conclude $u \sqsubset X$. Also, if $u \not\sqsubset X$, we can semidecide this by enumerating the forbidden words of $X$ by the assumption that $X$ is $\Pi^0_1$. The case that is left is that $u \sqsubset X$ and $u \not\sqsubset Y$ for any proper subshift $Y \subset X$.

In such a case, the union $Y$ of all proper subshifts of $X$ does not equal $X$. By Lemma~\ref{lem:Generator}, there exists a generator $w \sqsubset X$ such that $w \sqsubset x \in X$ implies that $x$ is weakly transitive in $X$. We may assume this $w$ is known to the algorithm, by way of a look-up table. Since forbidding $w$ results in a proper subshift of $X$, it must result in a subshift of $Y$. Thus, by compactness, there exists $n \in \N$ such that $x_{[0, |u|-1]} = u$ implies $x_{[i, i+|w|-1]} = w$ for some $i \in [-n, n]$, for all $x \in X$. This must also hold in a suitable SFT approximation of $X$, obtained by running the Turing machine defining $X$ for sufficiently many steps. Once such $n$ and an SFT approximation $Z$ is found, the algorithm can conclude $u \sqsubset X$. 
\end{proof}

\begin{remark}
An important point is that the generator $w$ was needed in the algorithm, and one cannot compute it directly from a Turing machine defining $X$ (as one can verify by an easy counterexample). Indeed, one can hardly compute anything from such a representation, because it cannot even be decided if the Turing machine defines an empty subshift. Thus, the result is not uniform in $X$. However, like in the case of Theorem~\ref{thm:MinimalRecursive} (where we supplied the information of whether the subshift is empty), the algorithm is at least somewhat uniform, as we only need the lists of forbidden words (the lists $(v_1, \ldots, v_k)$ given by Lemma~\ref{lem:SubshiftsPi}) defining the proper subshifts of $X$ and the lists of generators for all subshifts of $X$. An easy example shows that in the case of general quasiminimal subshifts, knowing whether the subshift is empty is \emph{not} enough to determine whether a given word $u$ is in the subshift. However, it is an interesting question exactly what ``shape information'' is needed.
\end{remark}

A trivial corollary of the previous result and Lemma~\ref{lem:SubshiftsPi} is that recursive quasiminimal subshifts are \emph{hereditarily recursive}, that is, all their subshifts are recursive.

\begin{corollary}
\label{cor:SubRecursive}
All subshifts of a recursive quasiminimal subshift $X$ are recursive.
\end{corollary}

Note that in the previous section, we essentially defined our $\Pi^0_1$ subshifts by giving a computational process that outputs a point, and we then showed that one has an algorithm for recognizing whether a given word occurs in the point. Just the fact that the point was given by a computational process does not automatically imply that there is a such an algorithm, but it does automatically follow that the subshift has at least a recursively enumerable language. Thus, we could not have hoped to find a $\Pi^0_1$ quasiminimal subshift which is not recursive with this technique, and indeed, by Theorem~\ref{thm:Recursive}, there is no such subshift. We can strengthen this observation slightly, and prove Proposition~\ref{prop:WeakCharacterization}, which states that all quasiminimal $\Pi^0_1$ subshifts are obtained as orbit closures of computable points. We first need a few lemmas.

\begin{lemma}
Let $X$ be a subshift whose language is recursively enumerable. Then the computable points are dense in $X$.
\end{lemma}

\begin{proof}
Given any word $u \sqsubset X$, one can construct a Turing machine that repeatedly extends $u$ by every possible letter in both directions and chooses the lexicographically minimal extension that still leads to a word that occurs in $X$. This gives the point and algorithm required in the definition of computability. 
\end{proof}

\begin{lemma}
If $X$ is a $\Pi^0_1$ quasiminimal subshift, then it is a finite union of orbit closures of computable points.
\end{lemma}

\begin{proof}
Suppose that the claim fails for some $\Pi^0_1$ quasiminimal subshift $X$. Enumerate the words of $X$ as $u_1, u_2, \ldots$. Let $x_1$ be a arbitrary computable point in $X$. Inductively, let $u_{k(j)}$ be the first word on the list $u_1, u_2, \ldots$ such that $u_{k(j)} \not\sqsubset Y_{j-1} = \OC{x_1} \cup \OC{x_2} \cup \cdots \cup \OC{x_{j-1}}$. Such $u_{k(j)}$ can always be found, since otherwise $X = Y_{j-1}$, a contradiction. Let then $x_j$ be a computable point containing $u_{k(j)}$. It follows that $Y_1 \subsetneq Y_2 \subsetneq \cdots$ is an infinite increasing sequence of subshifts of $X$, which contradicts quasiminimality. 
\end{proof}

The following proposition is subtly stronger.

\begin{proposition}
\label{prop:WeakCharacterization}
If $X$ is a $\Pi^0_1$ quasiminimal subshift, then there exists a finite set $\{x_1, x_2, \ldots, x_k\}$ of computable points with distinct languages such that for every $x \in X$ we have $\B(x_i) = \B(x)$ for some $i \in [1, k]$.
\end{proposition}

\begin{proof}
A $\Pi^0_1$ quasiminimal subshift has finitely many subshifts, which are all $\Pi^0_1$ and quasiminimal. We apply the previous lemma to each of them, and take the union of the finitely many sets of points obtained. To obtain distinct languages, repeatedly remove $x_j$ from the set if $\B(x_i) \neq \B(x_j)$ for some $i \neq j$. Now, suppose that there is a point $x \in X$ such that $\B(x_i) \neq \B(x)$ for all $i \in [1, k]$. Then, for each $i$, either $x$ contains a word $u_i$ such that $u_i \not\sqsubset x_i$, or $x_i$ contains a word $v_i$ such that $v_i \not\sqsubset x$. Define a subshift $Y$ of $X$ by forbidding the $v_i$, and choose a subword $w$ of $x$ containing all the $u_i$. Then, $x$ is a point of $Y$ containing a word $w$ which is not in any of the $x_i$. This is a contradiction, since we applied the previous lemma also to $Y$. 
\end{proof}

This proposition suggests a small tangent to the discussion. The following proposition shows that the previous result is not a characterization of $\Pi^0_1$ quasiminimal subshifts, and also shows that there is no symmetric version of Theorem~\ref{thm:MinimalRecursive} for ``$\Sigma^0_1$ subshifts''. The result is presumably well-known, but we are not aware of a reference for it.

\begin{proposition}
There exists a computable point $x \in \{0,1\}^\Z$ such that $\OC{x}$ is minimal and not $\Pi^0_1$.
\end{proposition}

\begin{proof}
We only need to describe a minimal subshift whose subwords are recursively enumerable, as such a subshift contains a computable point, and every point in a minimal subshift generates it.

Let $W_0 = \{0, 1\}$. For each $i$, we will define $W_i$ to be a set of two words of the same length such that every (bi-)infinite concatenation of words of $W_i$ has a unique partition into words of $W_i$, and both words of $W_i$ are contained in both words of $W_{i+1}$. For any list of sets with these properties, 
we obtain a minimal subshift by taking the limit points of the words $W_i$ in the obvious way. We denote by $\ell_i$ be maximal number such that for some $w \in W_i$, the word $w^{\ell_i} \in X$ occurs in some concatenation of words of $W_{i+1}$. We make sure $\ell_i$ is always finite, and strictly smaller than the length of the words $w \in W_{i+1}$. We will not define these words by a direct induction, but by a computational process, and the process may change the definition of $W_i$ arbitrarily late in the process of constructing these sets -- the reader should think of the $W_i$ as variables in the programming sense of the word in the following.

The property of unique partitions is clearly true for $W_0 = \{0, 1\}$. If $W_i = \{w_1, w_2\}$ (where $w_1 < w_2$ in lexicographical order) satisfies this property, and $u$, referred to as the \emph{cover-up}, is any concatenation of words in $W_i$, then we set $W_{i+1} = \{w_1^{i} w_2^{q-i} w_1 u \;|\; i \in \{1, 2\}\}$ where $q = 2|u|+3$. These words will have the property of unique partitions, contain both words of $W_i$, and certainly $\ell_i < |w_1 w_2^{q-1} w_1 u|$, as required. Note also that $\ell_i > |u|$, and if $u$ is the empty word $\lambda$, then $\ell_i = 3$.

Now, we describe the algorithm that constructs the sets $W_{i+1}$. One $i$ at a time, the algorithm sets $W_{i+1} = \{u_1, u_2\}$, where $u_1$ and $u_2$ are any two words constructed as above from the words of $W_i$ and the cover-up $u = \lambda$. Every time it computes a new set $W_i$, it outputs the words of $W_i$, and all their subwords, in lexicographical order. While constructing these words, the algorithm dovetails the computation of all Turing machines $T_1, T_2, \ldots$. If the $k$th machine halts and the machine has computed the sets $W_i$ for $i \leq K$ (where we may assume $K > k$), then the algorithm changes the definition of $W_k$: If $W_{k-1} = \{w_1, w_2\}$ and $W_K = \{u_1, u_2\}$ (and $w_1 < w_2$ and $u_1 < u_2$ in lexicographical order), let the cover-up be $u = u_1u_2$, and define $W_k$ as in the previous paragraph, so that every word of the new set $W_k$ contains the word $u$. Furthermore, since we previously had $\ell_{k-1} \leq |w|$ for $w \in W_k$, the new value of $\ell_{k-1}$ is strictly larger, as $\ell_{k-1} > |u|$, $u = u_1u_2$ and $u_i \in W_K$.

The crucial point is that since the algorithm changed its mind about the contents of $W_k$, all the subwords of words in $W_i$ for $i \geq k$ it has output sofar have been wrong, in the sense that the values of the sets $W_i$ were not correct -- however, the new value of $W_k$ conveniently covers up these lies\footnote{Of course, the \emph{words} that have been output are not really lies, in the sense that they are indeed words of the subshift $X$ we eventually obtain. However, the reader can probably agree that the algorithm has at least not been outputting subwords in a very honest order.} due to the choice of $u$: every word it has output sofar is in fact a subword of every word in the updated set $W_k$. Thus, the algorithm can forget the sets $W_i$ for $i > k$ it had computed previously, and continue its computation from the new set $W_k$.

It is clear that the sets $W_j$ all converge, as once those Turing machines $T_i$ with $i \leq j$ have halted which will eventually halt, the contents of $W_j$ will never change. The sets obtained in the limit still satisfy the properties we asked for, and thus the limit of these sets of words is a minimal subshift by a standard argument. The subshift has a recursively enumerable language by construction.

Finally, let us show the language of the subshift is not $\Pi^0_1$. For this, simply observe that if we had a decision algorithm for the language of $X$, we could decide, for each $W_i$ one by one, whether the machine $i$ ever halts. Namely, assuming we know whether the machine $T_i$ halts for all $i < j$, we can determine whether $T_j$ eventually halts as follows: we run the algorithm described above until it has computed the sets $W_i$ for all $i < j$ (and recomputed them as many times as necessary, whenever the machines $T_i$ with $i < j$ halt). Let $W_{j-1} = \{w_1, w_2\}$, in lexicographical order. Now, the machine $T_j$ eventually halts if and only if $w_2^4$ occurs in the language of $X$. 
\end{proof}

We note that it does not follow from the previous proof that the language of the subshift constructed is $\Sigma^0_1$-complete for many-one reductions, only Turing reductions.

We can strengthen Theorem~\ref{thm:Recursive} further, and show that even the halting problem of a quasiminimal subshift is decidable (compare this with Proposition~\ref{prop:OneMinimal}). More precisely, we show that the model-checking problem for elementary piecewise testable languages is decidable, and the halting problem follows as a special case. First, we state the obvious.

\begin{lemma}
\label{lem:IsFinite}
The generating poset and subpattern poset of a quasiminimal subshift are finite.
\end{lemma}

\begin{proof}
Let $X$ be quasiminimal. If $x \in X$, then $Y = \OC{x}$ is a subshift of $X$ with $\B(Y) = \B(x)$, and $X$ has finitely many subshifts, so the subpattern poset is finite. Let $\{x_1, x_2, \ldots, x_k\} \subset X$ be a finite set of representatives such that $\forall x \in X: \exists i: \B(x) = \B(x_i)$.

The inequality $u \leq_X v$ means that whenever $v \sqsubset x$ for $x \in X$, then also $u \sqsubset x$. Associate to each $u \in \B(X)$ the tuple $B(u) = (b_1, b_2, \ldots, b_k) \in \{0, 1\}^k$ where $b_i = 1 \iff u \sqsubset x_i$. If $B(u) = B(v)$, then $u \leq_X v \leq_X u$ holds, so the generating poset of $X$ is finite. 
\end{proof}

\begin{theorem}
\label{thm:TupleTheorem}
If $X \subset S^\Z$ is a quasiminimal $\Pi^0_1$ subshift, then the model-checking problem of $X$ for elementary piecewise testable languages is decidable.
\end{theorem}

\begin{proof}
Restated in terms of clopen sets, we are given $[C_1], [C_2], \ldots, [C_\ell]$, and we need to decide whether all of these are visited in order by a single point. We can deal with each of the finitely many tuples $T = (t_1, \ldots, t_\ell) \subset (S^*)^\ell$ with $t_i \in C_i$ separately, and check for such a tuple whether there exist $x \in X$ and $i_1 < \ldots < i_\ell \in \Z$ such that $x_{[i_j, i_j + |t_j|-1]} = t_j$ for all $1 \leq j \leq \ell$. We denote this condition by $T \sqsubset X$.

Let $w_1, w_2, \ldots, w_k$ be a finite set of representatives for the elements of the generating poset. To each word $u$, we associate $h_u \in \N$ such that in every extension of $u$ of length $h_u$, the unique $w_i$ with $w_i \sim u$ occurs, given by Lemma~\ref{lem:LeqSemidecidable}.
 
We may suppose these words, and their mutual ordering, are known to the algorithm. Now, let $T = (t_1, \ldots, t_\ell)$ be given. We may again assume $\QW(X) > 1$, and prove the decidability of $T \sqsubset X$ by induction on $\QW(X)$. We do a similar case analysis as in Theorem~\ref{thm:Recursive}. By the induction hypothesis, for any proper subshift $Y \subsetneq X$ it is decidable whether $T \sqsubset Y$, so we suppose $T \not\sqsubset Y$ for all proper subshifts $Y$ of $X$. If $X$ is the union of its proper subshifts, the algorithm can decide $T \sqsubset X$ by checking whether $T \sqsubset Y$ for all $Y \subsetneq X$, so we assume this is not the case. As above, denote by $Y \subsetneq X$ the union of all proper subshifts of $X$ and let $w \in S^*$ be a generator for $X$ (known to the algorithm). Now, choose a point $x \in X$ which contains $w$, so that $x$ is weakly transitive in $X$. 
We may assume all of the words $t_i$ appear in $x$, since if $t_i \not\sqsubset x$, then $t_i \not\sqsubset X$, which the algorithm can prove by enumerating forbidden patterns of $X$.

If $w$ appears infinitely many times in $x$ (which we may assume is known to the algorithm), the algorithm only needs to prove $t_i \leq_X w$ for all $i$ (which is possible, since $\leq_X$ is semidecidable, and $w$ is a generator): then the words $t_i$ in fact appear in $x$ in every possible order, and thus $T \sqsubset X$. Suppose then that $w$ only appears once. Then $x$ is an isolated point, so that the only points with the same subwords as $x$ are in the orbit of $x$. Furthermore, $w$ is an isolating pattern. Let $Y_L$ and $Y_R$ be the subshifts generated by the left and right tail of $x$.

Now, let $u = t_i$ for some $i$, let $u \sim w_j$, and let $h = h_u$. 
Now, if the unique $w_j$ with $w_j \sim u$ occurs in $Y_L$, then $u$ appears infinitely many times in the left tail of $x$. Otherwise, we can compute a bound on how far to the left of $w$ the word $u$ can appear: $w_j$ occurs at most $m$ steps to the left of $w$ for some $m$ (which we may assume is known to the algorithm by way of a look-up table), and then $u$ cannot appear $m + h_u$ steps to the left of the origin. Similarly, we can semidecide that $u$ appears infinitely many times to the right, or find a bound on how far to the right of $w$ it can appear. Since the language of $X$ is recursive, we can now easily decide whether $T \sqsubset X$. 
\end{proof}

The halting problem is the subcase where there are only two clopen sets.

\begin{corollary}
\label{cor:Halting}
The halting problem of a quasiminimal $\Pi^0_1$ subshift is decidable.
\end{corollary}

We give a simpler direct proof of this as well, using more external information. The simpler proof works for any fixed tuple length (with larger and larger look-up tables), but not in general.

\begin{proof}[Direct proof of Corollary~\ref{cor:Halting}]
Let $w_1, \ldots, w_k$ be representatives of the generating poset, and let $H(i, j)$ be the answer of the halting problem for the cylinders $[w_i], [w_j]$. Given any two words $u, v$, we find $w_i$ and $w_j$ such that $u \sim w_i$ and $v \sim w_j$. Using the fact that $X$ is recursive, we can check whether some $x \in X$ moves from $[u]$ to $[v]$ in at most
\[ |w_i| + h_{u, w_i} + h_{w_i, u} + |w_j| + h_{v, w_i} + h_{w_i, v} \]
steps (and if yes, answer "yes"). If not, the answer to the halting problem is $H(i, j)$. 
\end{proof}

\begin{remark}
By Theorem~\ref{thm:TupleTheorem}, we can decide $T \sqsubset X$ for any tuple $T$, and thus we can certainly decide $P(P_1(X), P_2(X), \ldots, P_k(X))$ where $P$ is an arbitrary Boolean formula and the $P_i(X)$ are statements of the form $T_i \sqsubset X$ for tuples of words $T_i$. In contrast, we saw in Theorem~\ref{thm:CountablePiecewiseTestableLanguages} that for piecewise testable languages, the Boolean closure of elementary piecewise testable languages, the model-checking problem is undecidable. This may seem contradictory, but it only means that
\[ P((\exists w \sqsubset X: P_1(w)), (\exists w \sqsubset X: P_2(w)), \ldots, (\exists w \sqsubset X: P_k(w))) \]
is decidable for $\Pi^0_1$ quasiminimal subshifts, while
\[ \exists w \sqsubset X: P(P_1(w), P_2(w), \ldots, P_k(w)) \]
is undecidable, where we write $P_i(w) \iff T_i \sqsubset w$.
\end{remark}

\subsection{Quasiminimal subshifts with $\N$-actions}
\label{sec:NActions}

Most of our study has been about quasiminimal subshifts $X \subset S^\Z$ with the $\Z$-action given by the shift map. We now show that quasiminimal $\N$-subshifts are in fact decidable (when subsystems are not required to be surjective). See Figure~\ref{fig:NCharacterization} for an illustration.

\begin{theorem}
\label{thm:NCharacterization}
Let $X \subset S^\N$ be a subshift. Then $X$ is quasiminimal if and only if there exist finitely many minimal subshifts $Y_1, Y_2, \ldots, Y_k$, finitely many points $x_1, \ldots, x_\ell$ and $m$ such that $X = \bigcup_i Y_i \cup \{x_1, \ldots, x_\ell\}$ such that for all $i \in [1, \ell]$, either $\sigma(x_i) \in \bigcup_i Y_i$ or $\sigma(x_i) \in \{x_{i+1}, \ldots, x_\ell\}$.
\end{theorem}

\begin{proof}
The "if" direction is easy to verify. Suppose then that $X$ is quasiminimal. We show that $X$ is the of required form.

Since each $\sigma^i(X)$ is a subsystem, we must have $\sigma^{m+1}(X) = \sigma^m(X)$ for some $m$, and we write $Y = \sigma^m(X)$. We claim that there are only finitely many points outside $Y$. Namely, if $x \in X$, then $\sigma^m(x) \in Y$. If $\ell$ is the least number such that $\sigma^\ell(x) \in Y$, then
\[ \OC{x} = \{x, \sigma(x), \ldots, \sigma^{\ell-1}(x)\} \cup Y' \]
is a subshift of $X$ for some subshift $Y' \subset Y$ (it is closed as a union of $\ell + 1$ closed sets, and shift-invariant by definition). Thus, each point $x \in X \setminus Y$ is contained in a subshift of $X$ which contains only finitely many points of $X \setminus Y$. If $X \setminus Y$ were infinite, this would imply that $X$ is not quasiminimal.

For $x, y \in X \setminus Y$, write $x \prec y$ if $\sigma^i(x) = y$ for some $i \in \N$. Since $\sigma^m(x) \in Y$, we cannot have $x \prec y \prec x$ unless $x = y$, so $\prec$ is a partial order on $X \setminus Y$. We obtain the points $x_1, \ldots, x_\ell$ by extending it to a total order.

Let $Y_1, Y_2, \ldots, Y_k$ be the minimal subsystems of $X$. The list is finite by quasiminimality. We now show that every point $x \in Y$ is uniformly recurrent, thus in $\bigcup_i Y_i$. Suppose this is not the case. Then there exists a word $w \sqsubset x$ which either occurs finitely many times in $x$, or occurs infinitely many times but with arbitrarily long gaps. In the first case, $y = \sigma^i(x)$ contains only one copy of $w$ at $y_{[0,|w|-1]} = w$, and in the second, we find a limit point $y \in \OC{x}$ with this property. Since $y \in Y$, then by a compactness argument it has an infinite chain of preimages
\[ \cdots \overset{\sigma}{\mapsto} y_4 \overset{\sigma}{\mapsto} y_3 \overset{\sigma}{\mapsto} y_2 \overset{\sigma}{\mapsto} y_1 = y \]
with $y_i \in Y$ for all $i$. Then $X_i = \OC{y_i}$ is a subshift of $Y$ containing exactly $i$ points where $w$ occurs, and thus the $X_i$ are an infinite family of subshifts of $X$, contradicting quasiminimality. 
\end{proof}

\begin{figure}
\begin{center}
\begin{tikzpicture}
\tikzset{>=stealth'}

\draw (0,0) ellipse (1cm and 1cm);
\node at (0,-0.4) {$Y_1$};
\draw (3,0) ellipse (1cm and 1cm);
\node at (3,-0.4) {$Y_2$};
\draw (8,0) ellipse (1cm and 1cm);
\node at (8,-0.4) {$Y_3$};

\node[draw,shape=circle,inner sep=2] (An) at (70:0.6) {};
\node[draw,shape=circle,inner sep=2] (Bn) at ($ (An) + (70:1) $) {};
\node[draw,shape=circle,inner sep=2] (Cn) at ($ (Bn) + (70:1) $) {};
\node[draw,shape=circle,inner sep=2] (Dn) at ($ (Cn) + (115:1) $) {};
\node[draw,shape=circle,inner sep=2] (En) at ($ (Cn) + (25:1) $) {};
\draw[->] (Dn) -- (Cn);
\draw[->] (Cn) -- (Bn);
\draw[->] (Bn) -- (An);
\draw[->] (En) -- (Cn);

\node[draw,shape=circle,inner sep=2] (Fn) at (170:0.7) {};
\node[draw,shape=circle,inner sep=2] (Gn) at ($ (Fn) + (110:1) $) {};
\node[draw,shape=circle,inner sep=2] (Hn) at ($ (Fn) + (170:1) $) {};
\node[draw,shape=circle,inner sep=2] (In) at ($ (Fn) + (230:1) $) {};
\node[draw,shape=circle,inner sep=2] (Jn) at ($ (Hn) + (170:1) $) {};

\draw[->] (Gn) -- (Fn);
\draw[->] (Hn) -- (Fn);
\draw[->] (In) -- (Fn);
\draw[->] (Jn) -- (Hn);

\node[draw,shape=circle,inner sep=2] (Kn) at ($ (3,0) + (65:0.6) $) {};
\node[draw,shape=circle,inner sep=2] (Ln) at ($ (3,0) + (-10:0.5) $) {};
\node[draw,shape=circle,inner sep=2] (Mn) at ($ (Kn) + (65:1) $) {};
\node[draw,shape=circle,inner sep=2] (Nn) at ($ (Ln) + (-10:1) $) {};

\draw[->] (Mn) -- (Kn);
\draw[->] (Nn) -- (Ln);

\node[draw,shape=circle,inner sep=2] (On) at ($ (8,0) + (135:0.8) $) {};
\node[draw,shape=circle,inner sep=2] (Pn) at ($ (On) + (135:1) $) {};
\node[draw,shape=circle,inner sep=2] (Qn) at ($ (Pn) + (90:1) $) {};
\node[draw,shape=circle,inner sep=2] (Rn) at ($ (Pn) + (180:1) $) {};
\node[draw,shape=circle,inner sep=2] (Sn) at ($ (Qn) + (45:1) $) {};
\node[draw,shape=circle,inner sep=2] (Tn) at ($ (Qn) + (135:1) $) {};
\node[draw,shape=circle,inner sep=2] (Un) at ($ (Rn) + (225:1) $) {};
\node[draw,shape=circle,inner sep=2] (Vn) at ($ (Rn) + (135:1) $) {};

\draw[->] (Pn) -- (On);
\draw[->] (Qn) -- (Pn);
\draw[->] (Rn) -- (Pn);
\draw[->] (Sn) -- (Qn);
\draw[->] (Tn) -- (Qn);
\draw[->] (Un) -- (Rn);
\draw[->] (Vn) -- (Rn);
\end{tikzpicture}
\end{center}
\caption{A typical quasiminimal $\N$-subshift. The $Y_i$ are the minimal subshifts and the small circles are the points $x_i$ and their images in $\bigcup_i Y_i$.}
\label{fig:NCharacterization}
\end{figure}
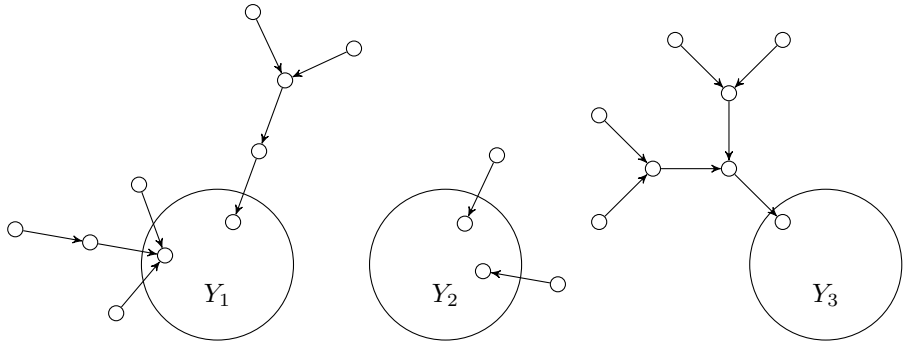

We see that the case of $\N$-actions looks very different than that of $\Z$-actions.

\begin{corollary}
\label{cor:NDecidable}
Let $X \subset S^\N$ be recursive and quasiminimal. Then $X$ is decidable.
\end{corollary}

\begin{proof}
Lemma~8 of \cite{DeKuBl06} states that a system whose limit set is a union of finitely many minimal systems is decidable. By the above theorem, it is clear that quasiminimal $\N$-subshifts have this property.
\end{proof}

\subsection{Decidability in the countable case}

We first show that, for pretty much trivial reasons, having only a single minimal subshift implies decidability in the countable case (even without assuming, a priori, that the subshift is $\Pi^0_1$). In particular, the subshift in Theorem~\ref{thm:CountableCodedLanguages} has the optimal number of subsystems.

\begin{proposition}
\label{prop:3IsDecidable}
Let $X \subset S^\Z$ be a countable subshift with $\QW(X) \leq 3$. Then $X = \OC{{^\infty u} v u^\infty}$ for some words $u, v$, and thus decidable.
\end{proposition}

\begin{proof}
The case $\QW(X) < 3$ is trivial, so suppose $\QW(X) = 3$. The CB-derivative $X'$ of $X$ is a proper subshift of $X$. We must have $\QW(X') < \QW(X)$, and thus $X'$ is a countable minimal subshift. This implies $X^{(2)} = \emptyset$, so the CB-rank of $X$ is 2. Every countable subshift with CB-rank 2 is generated by a single point which is eventually periodic in both directions \cite{CeDaToWy12,SaTo14a}, and since $X$ has only one minimal subshift $X'$, the repeating patterns in both tails must be equal. Decidability of such a subshift is easy to show.
\end{proof}



By Theorem~\ref{thm:CountablePiecewiseTestableLanguages}, there exists a recursive countable quasiminimal subshift $X$ for which the counting problem is $\Sigma^0_1$-complete. We needed $4$ proper nontrivial subsystems for this, that is, $\QW(X) = 6$. We show this is optimal. We begin with the following recoding argument.

\begin{definition}
\label{def:NormalForm}
Suppose $X$ is countable and quasiminimal. We say $X$ is in \emph{normal form} if the following holds. There exists a finite set of words $w_1, \ldots, w_n$, $w_{n+1}, \ldots, w_{n+m}$ such that
\begin{itemize}
\item the words $w_i$ are over disjoint alphabets and $\B_1(w_i) = |w_i|$ for all $w_i$,
\item every point $x \in X$ has a representation as a bi-infinite concatenation of the words $w_i$ (which is automatically unique),
\item for $j \in [1, n]$, the point $x_j = w_j^\Z$ is in $X$,
\item for each $j \in [n+1, n+m]$, there is (up to shifting) a unique point $x_j \in X$ with $w_j \sqsubset x_j$ but $w_k \not\sqsubset x_j$ for $k > j$, and $w_j$ occurs only once in $x_j$, and
\item every point of $X$ is in the orbit of one of the points $x_j$.
\end{itemize}
Furthermore, each of the points $x_j$ with $j \in [n+1, n+m]$ is of the following form: $x_j = x_j^L . w_j x_j^R$, where either $x_j^R = u_i^\N$ for some $i \in [1,k]$, or we have
\[ x_j^R = (w_{h_1})^{\ell_1} \; w_{k_1} \;  (w_{h_2})^{\ell_2} \;  w_{k_2} \;  (w_{h_3})^{\ell_3}  \; w_{k_3}  \; (w_{h_4})^{\ell_4} \ldots, \]
where for all $i$, $h_i \in [1, n]$, $k_i \in [n+1, n+j-1]$ and $\ell_i > 1$
and symmetric conditions hold for $x_j^L$.
\end{definition}

It is easy to see that a subshift $X$ in normal form with fixed $n, m$ contains precisely $n$ minimal subsystems.

\begin{lemma}
\label{lem:NormalForm}
Let $Y$ be countable and quasiminimal. Then $Y$ is conjugate to a subshift in normal form.
\end{lemma}

\begin{proof}
We construct the conjugacy inductively on $\QW(Y)$, so that if $Z \subset Y$ and $Z \subset Y'$, then the conjugacies from $Y$ and $Y'$ to their recoded versions $X$ and $X'$ have the same restriction to $Z$. This is guaranteed if we always directly modify the conjugacy obtained for subsystems when building the conjugacy for the full system in the induction step.

If $Y$ is countable and minimal, then it is clearly already of the form $^\infty w ^\infty$ as required ($w_1 = w$, $n = 1$ and $m = 0$). If $Y$ is the union of its proper subsystems, then we apply the inductive assumption to the subsystems, and observe that the conjugacy extends to the full system because the conjugacies agree in the intersections.\footnote{More precisely, this follows from the Pasting Lemma.}

Suppose then that $Y$ is not minimal, and not a union of proper subsystems. Then it contains a generator $w$, which we may assume is an isolating pattern, because isolated points are dense in a countable subshift. Let $y$ be the unique point containing $w$. Then $w$ occurs only finitely many times in $y$, since $Y$ is not minimal. Let $\phi$ be the conjugacy obtained for the union $Z$ of proper subsystems of $Y$. Suppose $\phi$ has radius $r$, and suppose $t$ is the maximum length of the words $w_i$ obtained for $Z$. There exists $c > 2r+t$ such that $y_{(-\infty, -c]}$ and $y_{[c, \infty)}$ both extend to points of the $(2r+t)$th order SFT approximation of $Z$.

This lets us apply $\phi$ in the tails of $y$: we define $\psi$ by $\psi(x)_i = \phi(x)_i$ if $w \not\sqsubset x_{[i-c-|w|,i+c+|w|]}$, and $\psi(x)_i = a_j$ if $x_{[i+j,i+j+|w|-1]} = w$ for $j \in [-c-|w|, c+1]$, where the $a_j$ are new symbols not used by $\phi$. Then $\psi$ is of the desired form, up to possibly extending the words surrounding the new word over the $a_j$ to full ones (there are unique such extensions since the words $w_j$ are over disjoint alphabets and $|\B_1(w_j)| = |w_j|$). We increase $m$ by $1$, let $w_{n+m}$ be the new word over the symbols $a_i$, and $x_{n+m} = \psi(y)$. 
\end{proof}

The conditions on the forms of points in Lemma~\ref{lem:NormalForm} of course do not automatically guarantee that the subshift is quasiminimal -- for this, one needs strong additional properties for the sequences $h_i, \ell_i, k_i$ in the left and right tails of the points $x_j$. The precise conditions for this are somewhat complicated. We note that in particular typically the gap sequence $\ell_i$ does not tend to infinity. This is because the tail of a point $x_j$ may occasionally get close to another point $x_k$ with $n < k < j$ (which happens if $\OC{x_k} \subset \OC{x_j}$). In such a case, one will see some short gaps $\ell_i$, since there are typically short gaps in the representation of $x_k$. Conversely, when we see short gaps $\ell_i$ (far enough away from the central pattern $u_j$), this must mean some word $u_k$ must be nearby, where $n < k < j$.


\begin{theorem}
\label{thm:Countable5Starfree}
Suppose $X$ is countable and $\QW(X) \leq 5$. Then the model-checking problem of starfree languages is decidable for $X$.
\end{theorem}

\begin{proof}
We perform a small case analysis to reduce to the interesting case that the lattice of subshifts of $X$ is isomorphic to $(\{0,1,2,3,4\}, <)$. As usual, in a model-checking problem we may assume $X$ is not the union of its proper subshifts. Thus, there exists a generator $w \sqsubset x$, which we may take to be isolating. By forbidding $w$ from $X$, we obtain a proper subshift $Y$ with $\QW(Y) \leq 4$.

First, suppose $Y$ is the union of proper subshifts, then we must have $Y = \B^{-1}(a^* + b^*)$ for some (possibly equal) words $a, b$. Namely, if $Y$ contains at least two minimal subsystems $\B^{-1}(a^*)$ and $\B^{-1}(b^*)$, it contains their union (because $Y \subset \OC{x}$), which already gives $4$ subsystems. It is impossible for $Y$ to contain only one minimal subsystem $\{^\infty a ^\infty\}$, as then $Y$ would contain at least two distinct points left- and right-asymptotic to ${^\infty}a^\infty$, and thus their union, which gives $5$ subsystems. If $Y = \B^{-1}(a^* + b^*)$, it is easy to see that $x$ is left asymptotic to ${^\infty}a^\infty$ and right asymptotic to ${^\infty}b^\infty$ or vice versa, and it follows that $X$ is decidable.

Suppose then that $Y$ is not a union of proper subshifts, and thus $Y = \OC{y}$. Since the maximal proper subshift $Z$ of $Y$ satisfies $\QW(Z) \leq 3$, we have $Z = \B^{-1}(u^* v u^*)$ for some words $u, v$. Thus, the subpattern poset of $Z$ is isomorphic to $(\{0, 1, 2\}, <)$, that of $Y$ is isomorphic to $(\{0, 1, 2, 3\}, <)$, and that of $X$ is isomorphic to $(\{0, 1, 2, 3, 4\}, <)$. In particular, $X$ has only one minimal subsystem, so we may assume $X$ is in normal form with $n = 1, m = 3$, since the hardness of model-checking problems is preserved under conjugacy.

Let $L$ be a given starfree language. We may suppose $L = S^* L S^*$, as this does not change the answer to the model-checking problem, and $S^* L S^*$ is starfree whenever $L$ is. Using the algebraic characterization of starfree languages, let $p$ be such that for all words $u$, $u^p \sim u^{p+1}$ in the syntactic monoid of $L$. We show that we can compute the central pattern and repeating patters for an eventually periodic point $x'$ such that some $u \in L$ occurs in $x'$ if and only if some $v \in L$ occurs in $x_n$, which gives the answer since model-checking problems are easy to solve for eventually periodic points, and $x_n$ is a weakly transitive point for $X$.

The important general observations here are the following: First, if $x = y_L u^{p+1} y_R$ for any word $u$ and any tails $y_L$ and $y_R$, then due to the assumption $L = S^* L S^*$ and the choice of $p$, the point $x' = y_L u^p y_R$ satisfies the model-checking problem for $L$ if and only if $x$ does. Furthermore, if $x = \lim_i y_i$, then $x$ satisfies the model-checking problem for $L$ if and only if one of the $y_i$ does by the definition of a limit, since the model-checking problem is about belonging to a particular type of open set. Thus, we may repeatedly contract powers of words in any point of $x$ and pass to a limit, and (a shift of) the resulting limit point will satisfy the model-checking problem for $L$ if and only if $x$ does. 

We explain how to perform the contraction process to the right tail of $x_4$ (in the notation of Definition~\ref{def:NormalForm}), so that we obtain an eventually periodic point in the limit. The left tail is handled symmetrically, to obtain the eventually periodic point for which we then solve the model-checking problem of $L$. Let $j$ be maximal such that $w_j$ occurs in $x_4$ infinitely many times to the right. Our contraction process contracts the right tail of $x_4$ to $w_4(uw_jv)^\infty$ for suitable words $u, v$.

If $j = 1$, then $x_4$ is already eventually periodic to the right, and we are done. Suppose then that $j > 1$. Since $w_4$ is an isolating pattern, we have $j < 4$. We may assume no $w_r$ with $r > j$ appears to the right of the central pattern $w_4$ of $x_4$ by modifying a finite part of the conjugacy (or otherwise restricting our attention to a suitable tail). 

Consider the case $j = 2$. Then, letting $h_i, \ell_i, k_i$ be as in Definition~\ref{def:NormalForm} for the point $x_4$ (whence $h_i = 1$ for all $i$), in fact $k_i = 2$ for all $i$ and $\ell_i \rightarrow \infty$ as $i \rightarrow \infty$. In this case, we can compute $s$ such that $w_2 w_1^{p - i'} w_2$ for $i' > 0$ cannot occur in the right tail of $x_4$ beyond the coordinate $s$. Such $s$ exists, because otherwise there is a right limit point of $x_4$ where $w_2$ occurs twice, but $w_3$ and $w_4$ do not occur, and $X$, being in normal form, does not contain such a point. We can compute $s$ as in the proof of Theorem~\ref{thm:TupleTheorem} using Lemma~\ref{lem:LeqSemidecidable}. Now, we simply contract, one by one, each $\ell_i$ with large enough $i$ (say, $i > s$) to $p$. The new right tail (to the right of $w_4$) is then of the form 
\[ x_4^R = w_1^{\ell_1} w_2 w_1^{\ell_2} w_2 w_1^{\ell_3} w_2 \cdots w_1^{\ell_s} w_2 w_1^p w_2 w_1^p w_2 w_1^p \ldots, \]
and we can compute such a new tail from the description of the language $L$.

The case that is left is $j = 3$. We apply the reasoning of the previous paragraph to $x_3$. We observe that in $x_3$, either the right tail is periodic with repeating pattern $w_1$, or we have $\ell'_i \rightarrow \infty$ where $\ell'_i$ is the sequence of gaps in the representation of $x_3$, and the corresponding claim is true also on the left. In particular, for some $h$ and $h'$, further than $h$ away from $w_3$ in $x_3$, no pattern $w_2 w_1^{p - i'} w_2$ for $i' > 0$ occurs, and in either $(x_3)_{[-h', -h]}$ or $(x_3)_{[h, h']}$, the word $w_2$ occurs at least $p$ times. The existence of $h$ follows as in the previous paragraph, and $h'$ exists because $\OC{x_2} \subset \OC{x_3}$.

We can now compute $s$ such that for each $i \geq s$, either $\ell_i, \ell_{i+1} > p$ and $k_i = 2$, or $k_{i'} = 3$ for some $i'$ with $|i' - i| < h$. Take $s$ further large enough that the distance between two occurrences of $w_3$ is at least $2h' + 2h$. Now, by the assumption on $h$ and $h'$, between any two occurrences of $w_3$, there are at least $p$ occurrences of $w_2$ separated by $w_1^{p + i'}$ for some $i' \geq 0$. We contract each $w_1^{p + i'}$ to $w_1^p$, and then the maximal pattern $(w_2 w_1^p)^t$ between the $w_3$ to $(w_2 w_1^p)^p$. Clearly, we obtain the same distance and intermediate word between any two occurrences of $w_3$, and we can compute this repeating pattern. 
\end{proof}

\subsection{Some open problems}
\label{sec:Opens}

We showed in Section~\ref{sec:Substitutions} that every substitution generates a subshift which has a decidable model-checking problem for regular languages. It seems likely that one can prove the decidability of model-checking for Muller automata similarly. While this would be an interesting result, one could also aim higher, at the model-checking problem for context-free languages. Is the model-checking problem for context-free languages decidable for subshifts generated by substitutions?

One can naturally extend context-free languages to infinite words, just like regular languages. As with regular languages, there are multiple ways to do this, and it is an interesting question whether their model-checking problems are always decidable for substitutive subshifts. We give a rough description of the approach of \cite{EsIv10}, where the \emph{Muller context-free languages} or MCFL, are defined. We take the family of countable finitely branching trees where each node carries a label from a finite set, the possible sets of children of each node are determined by a local rule, and for each infinite childward path, the set of labels occurring infinitely many times is an element of a prescribed set (the Muller condition). The language associated to it is the subset of countable totally ordered sets labeled by $S$ (written $S^\#$), defined by taking the induced ordering and labeling of the frontier, and restricting to words over terminal symbols.

For the model-checking problem, one should further intersect with $S^\Z$. Thus, we are interested in the model-checking problem for the $\mathrm{MCFL}$ languages contained in $S^\Z$. These languages are essentially bi-infinite concatenations of context-free languages governed by a Muller regular language.

It is known to be decidable whether a subshift generated by a substitution is finite. Recently, this was shown for the more general class of substitutive images of subshifts generated by substitutions \cite{Mi11,Du13}, called HD0L. The image of a quasiminimal subshift in a non-erasing substitution is quasiminimal, and thus these subshifts fit into our framework (at least in the case of syndetic long symbols). Which model-checking problems are decidable for this class?


In the topic of countable quasiminimal subshifts, it seems that the proof of Theorem~\ref{thm:Countable5Starfree} extends to show that in general, when the lattice of subshifts of $X$ is totally ordered and $X$ is countable and recursive, then its model-checking problem for starfree languages is decidable. However, the contraction process becomes harder to describe.

The condition $\QW(X) = 3$ automatically means $X$ is at least weakly transitive and its maximal proper subshift is minimal. Further, subshifts with $\QW(X) = 3$ can be seen to split into the following $4$ types:
\begin{enumerate}
\item $X$ is countable,
\item $X$ is uncountable but contains an isolated point,
\item $X$ contains no isolated point but its maximal proper subshift is finite, or
\item $X$ contains no isolated point and its maximal proper subshift is infinite.
\end{enumerate}
As Proposition~\ref{prop:3IsDecidable} shows, decidability questions are trivial for subshifts of type 1, and the proof of Theorem~\ref{thm:TransitiveLocallyTestableLanguages} shows that almost nothing is decidable for subshifts of type 4.

\begin{problem}
Which model-checking problems are decidable for subshifts of type 2 or 3?
\end{problem}

We conjecture that in the countable case, the model-checking problem is decidable for locally testable languages in general.

\begin{conjecture}
\label{con:LocallyTestableDecidable}
Let $X$ be a countable quasiminimal recursive subshift. Then the model-checking problem of $X$ for locally testable languages is decidable.
\end{conjecture}

Another interesting direction in the study of decidability when restricting to minimal systems, but considering the model-checking problem of harder languages. We gave an example of a minimal subshift whose model-checking problem is $\Sigma^0_1$-complete for context-free languages accepted by a deterministic pushdown automaton. For this, we used a subshift of the well-known Dyck shift. The high-level idea giving nonregularity in this example is the balancing of brackets, which inherently requires the use of a full stack. We do not know whether there are examples where the stack contains only unary symbols. A slightly more formal way to ask this question is the following.

\begin{question}
Is the model-checking problem of languages accepted by one-counter machines decidable for every minimal $\Pi^0_1$ subshift?
\end{question}

The precise definition of one-counter machines used in the question is left implicit -- it is part of the question, as we do not know what implications such choices have. The question is also interesting for \emph{strictly ergodic subshifts} -- minimal subshifts where words occur with well-defined frequencies. Here, one might (perhaps mistakenly) expect that on long enough words, the counter is decremented or incremented at a roughly constant rate. All we know about one-counter machines is that a variant of the language $a^n b^n$ is unlikely to yield a hard model-checking instance, since for any word $w$, $w^n$ can occur in an infinite minimal subshift for only finitely many $n$.

Of course, there are many natural classes in both formal language theory and complexity theory which fall between regular and context-free languages. Further decidability and undecidability results for the model-checking problem for minimal systems for any such class might be an interesting research direction.

\section*{Acknowledgements}

The author would like to thank Ilkka T\"orm\"a for detailed discussions on the topic. The author was partially supported by the Academy of Finland Grant 131558, CONICYT Proyecto Anillo ACT 1103, by Basal project CMM, Universidad de Chile and the FONDECYT project 3150552.

\bibliographystyle{plain}
\bibliography{../../bib/bib}{}

\end{document}